\documentclass{amsart}
\usepackage{amssymb, latexsym}
\usepackage{url}

\DeclareMathOperator{\Aut}{\mathrm{Aut}}

\DeclareMathOperator{\rank}{\mathrm{rank}}
\DeclareMathOperator{\Reg}{\mathrm{Reg}}
\DeclareMathOperator{\sgn}{\mathrm{sgn}}

\DeclareMathOperator{\Tor}{\mathrm{Tor}}
\DeclareMathOperator{\Vol}{\mathrm{Vol}}

\begin{document}
 \bibliographystyle{plain}

 \newtheorem{theorem}{Theorem}
 \newtheorem{lemma}{Lemma}
 \newtheorem{corollary}{Corollary}
 \newtheorem{conjecture}{Conjecture}
 \newtheorem{definition}{Definition}
 \newcommand{\mc}{\mathcal}
 \newcommand{\A}{\mc{A}}
 \newcommand{\B}{\mc{B}}
 \newcommand{\cc}{\mc{C}}
 \newcommand{\D}{\mc{D}}
 \newcommand{\E}{\mc{E}}
 \newcommand{\F}{\mc{F}}
 \newcommand{\G}{\mc{G}}
 \newcommand{\eH}{\mc{H}}
 \newcommand{\FN}{\F_n}
 \newcommand{\I}{\mc{I}}
 \newcommand{\J}{\mc{J}}
 \newcommand{\eL}{\mc{L}}
 \newcommand{\M}{\mc{M}}
 \newcommand{\eN}{\mc{N}}
 \newcommand{\qq}{\mc{Q}}
 \newcommand{\U}{\mc{U}}
 \newcommand{\X}{\mc{X}}
 \newcommand{\Y}{\mc{Y}}
 \newcommand{\eZ}{\mc{Z}}
 \newcommand{\C}{\mathbb{C}}
 \newcommand{\R}{\mathbb{R}}
 \newcommand{\N}{\mathbb{N}}
 \newcommand{\Q}{\mathbb{Q}}
 \newcommand{\Z}{\mathbb{Z}}
 \newcommand{\aA}{\mathfrak A}
 \newcommand{\bB}{\mathfrak B}
 \newcommand{\ff}{\mathfrak F}
 \newcommand{\uU}{\mathfrak U}
 \newcommand{\fb}{f_{\beta}}
 \newcommand{\fg}{f_{\gamma}}
 \newcommand{\gb}{g_{\beta}}
 \newcommand{\vphi}{\varphi}
 \newcommand{\bo}{\boldsymbol 0}
 \newcommand{\bone}{\boldsymbol 1}
 \newcommand{\ba}{\boldsymbol a}
 \newcommand{\balpha}{\boldsymbol \alpha}
 \newcommand{\bb}{\boldsymbol b}
 \newcommand{\bbbeta}{\boldsymbol \beta}
 \newcommand{\bm}{\boldsymbol m}
 \newcommand{\bgamma}{\boldsymbol \gamma}
 \newcommand{\bt}{\boldsymbol t}
 \newcommand{\bu}{\boldsymbol u}
 \newcommand{\bv}{\boldsymbol v}
 \newcommand{\bx}{\boldsymbol x}
 \newcommand{\bwy}{\boldsymbol y}
 \newcommand{\bxi}{\boldsymbol \xi}
 \newcommand{\bbeta}{\boldsymbol \eta}
 \newcommand{\bw}{\boldsymbol w}
 \newcommand{\bz}{\boldsymbol z}
 \newcommand{\whG}{\widehat{G}}
 \newcommand{\oK}{\overline{K}}
 \newcommand{\oKt}{\overline{K}^{\times}}
 \newcommand{\oQ}{\overline{\Q}}
 \newcommand{\oq}{\oQ^{\times}}
 \newcommand{\oQt}{\oQ^{\times}/\Tor\bigl(\oQ^{\times}\bigr)}
 \newcommand{\ot}{\Tor\bigl(\oQ^{\times}\bigr)}
 \newcommand{\h}{\frac12}
 \newcommand{\hh}{\tfrac12}
 \newcommand{\dx}{\text{\rm d}x}
 \newcommand{\dy}{\text{\rm d}y}
 \newcommand{\dmu}{\text{\rm d}\mu}
 \newcommand{\dnu}{\text{\rm d}\nu}
 \newcommand{\dla}{\text{\rm d}\lambda}
 \newcommand{\dlav}{\text{\rm d}\lambda_v}
 \newcommand{\trho}{\widetilde{\rho}}
 \newcommand{\dtrho}{\text{\rm d}\widetilde{\rho}}
 \newcommand{\drho}{\text{\rm d}\rho}
 \def\today{\number\time, \ifcase\month\or
  January\or February\or March\or April\or May\or June\or
  July\or August\or September\or October\or November\or December\fi
  \space\number\day, \number\year}

\title[Heights on groups]{Heights on groups and\\small multiplicative dependencies}
\author{Jeffrey~D.~Vaaler}
\thanks{This research was supported by the National Science Foundation, DMS-06-03282.}
\address{Department of Mathematics, University of Texas, Austin, Texas 78712 USA}

\email{vaaler@math.utexas.edu}
\numberwithin{equation}{section}

\begin{abstract}  We generalize the absolute logarithmic Weil height from elements of the multiplicative group
$\oQt$ to finitely generated subgoups of $\oQt$.  The height of a finitely generated subgroup is
shown to equal the volume of a certain naturally occurring, convex, symmetric subset of Euclidean space.
This connection leads to a bound on the norm of integer vectors that give multiplicative dependencies
among finite sets of algebraic numbers.
\end{abstract}

\maketitle

\section{Introduction}

Let $\oQ$ denote the field of algebraic numbers, and write $\oq$ for its multiplicative group of
nonzero elements. Let $\alpha_1, \alpha_2, \dots , \alpha_N$ be a finite, nonempty collection of 
elements in $\oq$, and assume that these elements generate a multiplicative subgroup
\begin{equation}\label{setup1}
A = \Big\{\prod_{n=1}^N \alpha_n^{\xi_n} : \bxi \in \Z^N\Big\}
\end{equation}
of positive, torsion free rank $M$.  If $k$ is an algebraic number field and $A \subseteq k^{\times}$, then the torsion subgroup $\Tor(A)$
is contained in the subgroup of roots of unity in $k^{\times}$, and therefore $\Tor(A)$ is a finite cyclic group.
We assume that $1 \le M < N$.  Then the elements $\alpha_1, \alpha_2, \dots , \alpha_N$ are
multiplicatively dependent, and the set of multiplicative dependencies
\begin{equation}\label{setup2}
\eZ = \Big\{\bz \in \Z^N : \prod_{n=1}^N \alpha_n^{z_n} \in \Tor(A)\Big\}
\end{equation}
is a $\Z$-module of positive rank $L = N-M$.  In this paper we consider the problem of showing that $\eZ$
contains a nonzero vector $\bz$ for which the norm 
\begin{equation*}\label{setup3}
|\bz|_{\infty} = \max\{|z_1|, |z_2|, \dots , |z_N|\}
\end{equation*}
is not too large.  More generally, we will show that $\eZ$ contains $L$ linearly independent vectors
$\bz_1, \bz_2, \dots , \bz_L$ such that the product
\begin{equation}\label{setup4}
\prod_{l=1}^L |\bz_l|_{\infty}
\end{equation}
is not too large.  We will establish bounds on the product (\ref{setup4}) which depend on
the absolute logarithmic Weil height of the numbers $\alpha_1, \alpha_2, \dots , \alpha_N$.  If
$A \subseteq k^{\times}$, then we give more explicit bounds for (\ref{setup4}) which depend
on the heights of the numbers $\alpha_1, \alpha_2, \dots , \alpha_N$ and also on the field $k$.  

Let $h:\oq \rightarrow [0, \infty)$ denote the absolute logarithmic Weil height. 
If $\alpha$ is an element of $\oq$ and $\zeta$ is an
element in the torsion subgroup $\ot$, then it is well known that $h(\alpha \zeta) = h(\alpha)$.  Let $\G = \oQt$
denote the corresponding quotient group.  It follows that 
\begin{equation*}\label{setup4.5}
h:\G\rightarrow [0, \infty)
\end{equation*}
is well defined on the cosets of $\G$.  We note that $\G$ has
the structure of a vector space (written multiplicatively) over the field $\Q$ of rational numbers.  
To see this let $r/s$ denote a rational number, where $r$ and $s$ are relatively prime integers 
and $s$ is positive.  If $\alpha$ is in $\oq$ and $\zeta_1$ and $\zeta_2$ are in $\ot$, then all roots of the 
two polynomial equations
\begin{equation*}
x^s - (\zeta_1 \alpha)^r = 0\quad\text{and}\quad x^s - (\zeta_2 \alpha)^r = 0
\end{equation*}
belong to the same coset in $\G$.  If we write $\alpha^{r/s}$ for this coset, we
find that 
\begin{equation*}
(r/s, \alpha)\mapsto \alpha^{r/s}
\end{equation*}  
defines a scalar multiplication in the abelian group $\G$.  It follows that $\G$ is a
vector space (written multiplicatively) over the field $\Q$.  We can also draw this conclusion from the
structure theory for the multiplicative group of an algebraically closed field (see \cite[Theorem 77.1]{fuchs1960}).

If $\alpha$ and $\beta$ are points in $\G$, then from basic
properties of the height (see \cite[section 1.5]{bombieri2006})
we find that $(\alpha, \beta)\mapsto h\bigl(\alpha \beta^{-1}\bigr)$ defines a metric on
the vector space $\G$.  If $r/s$ is an element of $\Q$ as above, and $\alpha$ belongs to $\G$, we have the well known 
identity (see \cite[Lemma 1.5.18]{bombieri2006})
\begin{equation*}\label{setup5}
h\bigl(\alpha^{r/s}\bigr) = |r/s|h(\alpha).
\end{equation*}
It follows that $\alpha \mapsto h(\alpha)$ defines a norm on the vector space $\G$ with respect to the usual
archimedean absolute value on the field $\Q$ of scalars.  In this paper we will make use of the completion of $\G$
with respect to this norm.  The completion is a Banach space over the field $\R$ of real numbers, and the structure
of this Banach space was determined in \cite[Theorem 1]{all2009}.

Rather than working with elements 
$\alpha_1, \alpha_2, \dots , \alpha_N$ in $\oq$, it will be convenient throughout the remainder 
of this paper to write $\alpha_1, \alpha_2, \dots , \alpha_N$ for the image of these points in $\G$.  And we write
\begin{equation}\label{setup6}
\aA = \Big\{\prod_{n=1}^N \alpha_n^{\xi_n} : \bxi \in \Z^N\Big\}
\end{equation}
for the subgroup of rank $M$ generated by $\alpha_1, \alpha_2, \dots , \alpha_N$ in $\G$.  Then the $\Z$-module $\eZ$ of 
multiplicative dependencies is given by
\begin{equation}\label{setup7}
\eZ = \Big\{\bz \in \Z^N : \prod_{n=1}^N \alpha_n^{z_n} = 1\Big\}.
\end{equation}

We may regard the logarithmic Weil height $h$ as defined on the collection of all finitely generated subgroups 
$\aA\subseteq \G$ having rank $1$.  For if $\aA$ has 
rank $1$, if $\aA$ is generated by the coset $\alpha_1$ and also by the coset $\alpha_2$ in $\G$, then we must have either
$\alpha_1 = \alpha_2$ or $\alpha_1^{-1} = \alpha_2$.  As $h(\alpha_1) = h(\alpha_1^{-1})$, we find that 
\begin{equation*}\label{setup20}
h(\aA) = h(\alpha_1) = h(\alpha_2)
\end{equation*}
determines a well defined height on the subgroup $\aA$.  In this paper we extend the absolute logarithmic 
Weil height to the collection of all finitely generated subgroups of $\G$ having arbitrary positive rank.  To accomplish 
this we will make use of the isometric isomorphism $\alpha \mapsto f_{\alpha}(y)$ that was identified in \cite[equation (5.1)]{all2009}.
In particular, this map associates each point $\alpha$ in $\G$ with a
certain continuous, compactly supported, real valued function $y\mapsto f_{\alpha}(y)$, defined on the locally compact
Hausdorff space $Y$ of all places of $\oQ$. 
Once the height has been extended to finitely generated subgroups of $\G$, as in (\ref{ht25}), we will show that every finitely 
generated subgroup contains multiplicatively independent elements such that the product of the heights of the independent
elements is majorized by the height of the subgroup. The precise result is as follows.

\begin{theorem}\label{thm1}  Let $\aA \subseteq \G$ be a finitely generated subgroup of positive rank $N$, and let the height 
$h(\aA)$ be defined by {\rm (\ref{ht25})}.
Then there exist multiplicatively independent elements $\beta_1, \beta_2, \dots , \beta_N$ in $\aA$ such that
\begin{equation}\label{setup25}
h(\beta_1) h(\beta_2) \cdots h(\beta_N) \le h(\aA).
\end{equation}
Moreover, the subgroup $\bB\subseteq \aA$ generated by $\beta_1, \beta_2, \dots , \beta_N$, has index at most $N!$ in $\aA$.
\end{theorem}

We will obtain Theorem \ref{thm1} from a more general result of the same sort which applies to finitely generated subgroups contained in the
completion of $\G$ with respect to the metric induced in $\G$ by the Weil height.  The main result in \cite{all2009} established an
isometric isomorphism between the completion of $\G$ and a closed subspace $\X$ contained in a 
certain real Banach space $L^1(Y, \B, \lambda)$.  Theorem \ref{thm4} and Corollary \ref{cor3} provide a generalization of Theorem
\ref{thm1} to arbitrary lattices of finite rank contained in an arbitrary $L^1$-space.  In particular, the more
general assertions of Theorem \ref{thm4} and Corollary \ref{cor3} make no reference to algebraic numbers.  The connection with
algebraic numbers and the space $\G$ arises from the results obtained in \cite{all2009}.

As an illustration of Theorem \ref{thm1}, we consider the special case in which $\aA$ is the image of a group of $S$-units.  Toward
this end, let $k$ be an algebraic number field, and let $S$ be a finite collection of places of $k$ such that $S$ contains 
all the archimedean places of $k$.  Write
\begin{equation}\label{setup30}
U_S(k) = \{\gamma \in k^{\times}: |\gamma|_v = 1\ \text{for all places}\ v\notin S\}
\end{equation}
for the multiplicative group of $S$-units in $k^{\times}$.  We assume that $|S| = s \ge 2$ so that $U_S(k)$ has positive, torsion free 
rank $s - 1$, and we write $\uU_S(k)$ for the image of $U_S(k)$ in $\G_k$.  Obviously, $\uU_S(k)$ also has 
positive rank $s - 1$.  

\begin{theorem}\label{cor2}  Let $k$ be an algebraic number field, $U_S(k)$ the multiplicative group of $S$-units in $k^{\times}$,
and write $\uU_S(k)$ for the image of $U_S(k)$ in $\G_k$.  If $|S| = s \ge 2$ then
\begin{equation}\label{setup31}
h\bigl(\uU_S(k)\bigr) = \dfrac{s! \Reg_S(k)}{\bigl(2 [k: \Q]\bigr)^{s-1}},
\end{equation}
where $\Reg_S(k)$ denotes the $S$-regulator of $U_S(k)$.
\end{theorem}

By combining Theorem \ref{thm1} and (\ref{setup31}), we find that $\uU_S(k)$ contains multiplicatively independent elements
$\gamma_1, \gamma_2, \dots , \gamma_{s-1}$, such that
\begin{equation}\label{setup32}
h(\gamma_1) h(\gamma_2) \cdots h(\gamma_{s-1}) \le \dfrac{s! \Reg_S(k)}{\bigl(2 [k:\Q]\bigr)^{s-1}}.
\end{equation} 
This sharpens similar inequalities due to Brindza \cite{brindza1991} and Hajdu \cite{hajdu1993}, and is of
essentially the same quality as that obtained by Bugeaud and Gy\"ory \cite[Lemma 1]{bugeaud1996}.  Of course (\ref{setup25})
applies more generally to finitely generated subgroups of $\G$ that do not necessarily occur as the image of a group of $S$-units.

Our second main result is an inequality that bounds both the product (\ref{setup4}) and the product of the heights 
of $M$ multiplicatively independent elements from the group $\aA$.  The upper bound depends only on the rank $M$ and on
the heights of the generators $\alpha_1, \alpha_2, \dots , \alpha_N$. 
Earlier results on problems of this sort were established by Baker \cite{baker1968} and 
Stark \cite[Lemma 7]{stark1972}.  Then sharper and more explict bounds were obtained in papers
of Loxton and van der Poorten \cite{loxton1983}, \cite{alf1976}, \cite{alf1977}.  Their
results provide a bound on the norm of one nonzero element in $\eZ$.  A somewhat simpler,  
explicit bound on the norm of one nonzero element in $\eZ$ has recently been given by 
Loher and Masser \cite[Corollary 3.2]{loher2004}, and attributed to K.~Yu.  The result we prove here
is roughly comparable to, but sharper and simpler than, inequalities of Matveev \cite[Theorem 4]{matveev1994}
and Bertrand \cite[Theorem 2]{bertrand1997}.

\begin{theorem}\label{thm2}  Let $\alpha_1, \alpha_2, \dots , \alpha_N$ be
elements of the vector space $\G$ which generate a subgroup $\aA$ of positive rank $M$.
If $1 \le M < N$ then there exist $L = N-M$ linearly independent elements $\bz_1, \bz_2, \dots , \bz_L$ in the 
$\Z$-module $\eZ$ defined by {\rm (\ref{setup7})}, and $M$ multiplicatively independent elements 
$\beta_1, \beta_2, \dots , \beta_M$ in $\aA$, such that 
\begin{equation}\label{setup8}
\bigg\{\prod_{l=1}^L |\bz_l|_{\infty}\bigg\}\bigg\{\prod_{m=1}^M h(\beta_m)\bigg\}\le \bigg\{\sum_{n=1}^N h(\alpha_n)\bigg\}^M.
\end{equation}
\end{theorem}

For each algebraic number field $k$, we write $\G_k$ for the image of $k^{\times}$ in $\G$ under 
the canonical homomorphism.  It follows that $\G_k$ is a subgroup of $\G$, and clearly $\G_k$ is isomorphic 
to the quotient group $k^{\times}/\Tor\bigl(k^{\times}\bigr)$.  For each positive integer $M$ we define
\begin{align}\label{setup10}
\begin{split}
\eH(k,M) = \inf\big\{&h(\gamma_1) h(\gamma_2) \cdots h(\gamma_M) : \gamma_1, \gamma_2, \dots , \gamma_M\\
                          &\ \text{are multiplicatively independent elements in $\G_k$}\big\}.
\end{split}                               
\end{align}
It follows from a well known result of Kronecker \cite{kronecker1857} that $0 < h(\gamma)$ for all points $\gamma \not= 1$
in $\G_k$.  And by Northcott's theorem \cite{northcott1949} there are only finitely many points $\gamma$ in $\G_k$ with 
$h(\gamma)$ less than a positive constant.  Hence the infimum on the right of (\ref{setup10}) is achieved, and 
in particular the value of $\eH(k, M)$ is positive.  
It is conjectured (see \cite[Conjecture 4.4.6]{bombieri2006}) that for each positive integer $M$ there exists a positive 
constant $c(M)$ such that
\begin{equation}\label{setup12}
\frac{c(M)}{[k: \Q]} \le \eH(k, M)
\end{equation}
for each algebraic number field $k$.  The existence of such a constant $c(M)$ in case $M = 1$, would provide an answer to a well
know question first posed by D.~H.~Lehmer \cite{lehmer1933}.  Then (\ref{setup12}) is a natural generalization of 
Lehmer's problem to higher dimensions.  An unconditional result in the direction of the
conjectured lower bound (\ref{setup12}) has been obtained by Amoroso and David \cite{amoroso1999}.  They prove that
for each positive integer $M$ there exists a positive constant $c_1(M)$ such that
\begin{equation}\label{setup14}
\frac{c_1(M)}{[k: \Q]} \bigl(\log (3[k: \Q])\bigr)^{-M \kappa(M)} \le \eH(k, M),
\end{equation}
where $\kappa(M) = (M+1)\{(M+1)!\}^M - 1$.  

If $\alpha_1, \alpha_2, \dots , \alpha_N$ belong to $\G_k$, then the number $\eH(k, M)$ defined by (\ref{setup10}) 
is clearly a lower bound for the second product that appears on the left of (\ref{setup8}).  This leads to the following 
inequality.

\begin{corollary}\label{cor1}  Let $k$ be an algebraic number field.  Let $\alpha_1, \alpha_2, \dots , \alpha_N$ be
elements of the multiplicative group $\G_k$ which generate a subgroup $\aA$ of positive rank $M$.
If $1 \le M < N$ then there exist $L = N-M$ linearly independent elements $\bz_1, \bz_2, \dots , \bz_L$ in the 
$\Z$-module $\eZ$ defined by {\rm (\ref{setup7})}, such that
\begin{equation}\label{setup17}
\prod_{l=1}^L |\bz_l|_{\infty} \le \eH(k, M)^{-1} \bigg\{\sum_{n=1}^N h(\alpha_n)\bigg\}^M.
\end{equation}
\end{corollary}

Under the hypotheses of Corollary \ref{cor1}, the inequalities (\ref{setup14}) and (\ref{setup17}) can be combined.  This plainly 
leads to the more explicit bound
\begin{equation}\label{setup18}
\prod_{l=1}^L |\bz_l|_{\infty} \le \frac{[k: \Q]}{c_1(M)}\bigl(\log (3[k: \Q])\bigr)^{M \kappa(M)} \bigg\{\sum_{n=1}^N h(\alpha_n)\bigg\}^M,
\end{equation}
where $c_1(M)$ is a positive constant that depends only on $M$.  An even sharper bound follows from very recent work of 
Amoroso and Viada \cite{amoroso2012}.

In section 2--4 of this paper we establish a new formula for the volume of a zonoid.  The precise result is stated in Theorem \ref{thm3}.
As this is a formula for the volume of a certain convex symmetric set in Euclidean space, we obtain a further result, stated
in section 5 as Theorem \ref{thm4}, from an application of Minkowski's second theorem.  Algebraic numbers appear again in section 6.  There we 
combine Theorem \ref{thm4} with the results of \cite{all2009} to complete the proof of Theorem \ref{thm1}.  Finally, in section 7 we combine
Theorem \ref{thm1} and the basis form of Siegel's Lemma obtained in \cite{bombieri1983}.  This leads to a proof of Theorem \ref{thm2}.

\section{Zonoids}  

Let $\delta:\R^N\rightarrow [0, \infty)$ be a norm, and 
\begin{equation*}\label{zon0}
B = \big\{\bxi \in \R^N: \delta(\bxi) \le 1\big\}
\end{equation*}
the associated unit ball in the $N$-dimensional, real Banach space $(\R^N, \delta)$.  Then $B$
is a compact, convex, symmetric subset having a nonempty interior.  The dual norm
$\delta^*:\R^N\rightarrow [0, \infty)$ is given by
\begin{equation}\label{zon1}
\delta^*(\bbeta) = \sup\big\{\langle \bxi, \bbeta\rangle: \delta(\bxi) \le 1\big\},
\end{equation}  
where we write
\begin{equation*}\label{zon2}
\langle \bxi, \bbeta \rangle = \xi_1\eta_1 + \xi_2\eta_2 + \cdots + \xi_N\eta_N
\end{equation*}
for the usual inner product on vectors $\bxi$ and $\bbeta$ in $\R^N$.  We also write
\begin{equation}\label{zon3}
B^* = \big\{\bbeta \in \R^N: \delta^*(\bbeta) \le 1\big\}
\end{equation}
for the dual (or polar) unit ball in the dual Banach space $(\R^N, \delta^*)$.  Of course,
$B^*$ is also a compact, convex, symmetric subset having a nonempty interior, $\delta$ is the norm
that is dual to $\delta^*$, and $B$ is the unit ball that is dual to $B^*$.

Let $\Vol_N$ denote $N$-dimensional Lebesgue measure or $N$-dimensional volume defined on Borel measurable subsets
of $\R^N$.  An important problem in the theory of convex sets is to estimate the volume product
$\Vol_N(B) \Vol_N(B^*)$ by an expression that depends only on the dimension $N$.  An upper bound is
given by the Blaschke-Santal\'o inequality (see \cite{santalo1949})
\begin{equation}\label{zon4}
\Vol_N(B) \Vol_N(B^*) \le \kappa_N^2,
\end{equation}
where $\kappa_N$ is the volume of the $N$-dimensional unit ball.  The lower bound
\begin{equation}\label{zon5}
\frac{4^N}{N!} \le \Vol_N(B) \Vol_N(B^*)
\end{equation}
was conjectured by K. Mahler \cite{mahler1939a}, \cite{mahler1939b}, who proved the
inequality (\ref{zon5}) in the special case $N = 2$.  In general, the conjectured lower bound
(\ref{zon5}) remains an open problem.  However, (\ref{zon5}) has been proved by S. Reisner 
\cite[Theorem 2]{reisner1985}, in case either $B$ or $B^*$ is a zonoid.  Related results
are discussed in \cite{lopez1998}, \cite{reisner1986}, and \cite{reisner1988}. 

We briefly recall basic facts that we will need concerning zonoids.  A more extensive discussion
of zonoids is contained in the survey papers of Bolker \cite{bolker1969},
Schnieder and Weil \cite{schneider1983}, and in the monograph of Thompson \cite{thompson1996}.  Let
\begin{equation*}\label{zon10}
\Sigma_{N-1} = \big\{\bxi \in \R^N: |\bxi|_2 = 1\big\}
\end{equation*}
denote the surface of the Euclidean unit ball in $\R^N$.  Here we write $|\bxi|_2$ for the usual Euclidean
or $l^2$-norm on vectors in $\R^N$.  We also write
\begin{equation*}\label{zon11}
|\bxi|_1 = |\xi_1| + |\xi_2| + \cdots + |\xi_N|\quad\text{and}\quad |\bxi|_{\infty} = \max\{|\xi_1|, |\xi_2|, \dots , |\xi_N|\}
\end{equation*}
for the $l^1$ and $l^{\infty}$ norms, respectively, on vectors $\bxi$ in $\R^N$.  Then the dual unit ball $B^*$ is a {\it zonoid}
(see \cite[Theorem 1.2]{schneider1983}) if and only if there exists a finite, even measure $\rho$, defined 
on the $\sigma$-algebra of Borel subsets of $\Sigma_{N-1}$, such that
\begin{equation}\label{zon13}
\delta(\bxi) = \hh \int_{\Sigma_{N-1}} \bigl|\langle \bxi, \bbeta\rangle\bigr|\ \drho(\bbeta)
\end{equation}
for all vectors $\bxi$ in $\R^N$.  We say that $\rho$ is an {\it even} (or {\it symmetric}) measure if $\rho(E) = \rho(-E)$ for
all Borel subsets $E\subseteq \Sigma_{N-1}$.  

Let
\begin{equation*}\label{zon14}
U = (u_{mn}) = \bigl(\bu_1\ \bu_2\ \cdots\ \bu_N\bigr)
\end{equation*}
be an $M\times N$ real matrix with $\rank(U) = N$, where $\bu_1, \bu_2, \dots , \bu_N$ denote the columns of $U$.
Then let
\begin{equation*}\label{zon15}
U^t = V = \bigl(\bv_1\ \bv_2\ \cdots\ \bv_M\bigr)
\end{equation*}
denote the $N\times M$ transpose of $U$, where $\bv_1, \bv_2, \dots , \bv_M$ are the columns of $V$.  It follows that the map
\begin{equation}\label{zon16}
\bxi \mapsto \bigl|U\bxi\bigr|_1= \sum_{m=1}^M~\Bigl|\sum_{n=1}^N u_{mn} \xi_n\Bigr| = \sum_{m=1}^M \bigl|\langle \bxi, \bv_m\rangle\bigr|
\end{equation}
defines a norm on $\R^N$.  The unit ball associated to the norm (\ref{zon16}) is obviously
\begin{equation*}\label{zon17}
B_U = \big\{\bxi \in \R^N: \bigl|U\bxi|_1 \le 1\big\},
\end{equation*}
and it is not difficult to show that the dual unit ball is given by
\begin{equation}\label{zon18}
B_U^* = \big\{V\bw : \bw \in \R^M\ \text{and}\ |\bw|_{\infty} \le 1\big\}.
\end{equation}
Thus the dual unit ball $B_U^*$ is the linear image in $\R^N$ of the unit cube in $\R^M$, where $N \le M$, and therefore 
$B_U^*$ is an example of a {\it zonotope}.  It is easy to see that a zonotope of the form (\ref{zon18}) is an example of a 
zonoid.  This follows because the sum on the right of (\ref{zon16}) can be written as
\begin{equation}\label{zon19}
\bigl|U\bxi\bigr|_1 = \sum_{m=1}^M \bigl|\langle \bxi, \bv_m\rangle\bigr| 
	= \hh \int_{\Sigma_{N-1}} \bigl|\langle \bxi, \bbeta\rangle\bigr|\ \drho_V(\bbeta),
\end{equation}
where $\rho_V$ is the unique atomic measure on Borel subsets of $\Sigma_{N-1}$ that assigns point mass $|\bv_m|_2$
to each of the $2M$ points $\pm \bv_m |\bv_m|_2^{-1}$ for $m = 1, 2, \dots , M$.

We will make use of a basic formula for the $N$-dimensional volume of the zonotope (\ref{zon18}) that was established by
P.~McMullen \cite{McMullen1984}.  For each subset $I\subseteq \{1, 2, \dots , M\}$ of cardinality $|I| = N$, let $V_I$ denote the 
$N\times N$ submatrix of $V$ such that $n = 1, 2, \dots , N$ indexes rows and $m$ in $I$ indexes columns.  Then McMullen's formula is
\begin{equation}\label{zon20}
\Vol_N\bigl(B_U^*\bigr) = 2^N \sum_{|I| = N} \bigl|\det V_I\bigr|.
\end{equation}

\section{Zonoids obtained from Banach spaces}

Let $\X$ be a Banach space with norm given by $\|\ \|_{\X}$, and let $\X^*$ denote the dual Banach space of continuous 
linear functionals on $\X$.  If $\bx$ is a vector in $\X$
and $\bx^*$ is a vector in $\X^*$, we write $\langle \bx, \bx^*\rangle$ for the value of the continuous linear 
functional $\bx^*$ at the point $\bx$.  Then the dual norm on $\X^*$ is given by
\begin{equation*}\label{zon29}
\|\bx^*\|_{\X^*} = \sup\big\{\langle \bx, \bx^*\rangle : \|\bx\|_{\X} \le 1\big\}.
\end{equation*}
If $\M\subseteq \X$ is a closed linear subspace we write
\begin{equation}\label{zon30}
\M^{\perp} = \big\{\bx^* \in \X^*: \langle \bm, \bx^*\rangle = 0\ \text{for all}\ \bm \in \M\big\}
\end{equation}
for the subspace of all continuous linear functionals in $\X^*$ that vanish on $\M$.  Then $\M^{\perp} \subseteq \X^*$ is a 
closed linear subspace, and we recall that the quotient norm on the Banach space $\X^*/\M^{\perp}$ is defined by
\begin{equation}\label{zon31}
\bigl\|\bx^* + \M^{\perp}\bigr\|_{\X^*/\M^{\perp}} = \inf\big\{\|\bx^* + \bwy^*\|_{\X^*}: \bwy^* \in \M^{\perp}\big\}.
\end{equation}
By the Hahn-Banach theorem (see \cite[Theorem 3.3]{rudin1991}), each linear functional $\bm^*$ in $\M^*$ extends to a linear
functional $\bz^*$ in $\X^*$.  If $\bz_1^*$ and $\bz_2^*$ are both extensions of $\bm^*$, then $\bz_1^* - \bz_2^*$ clearly belongs to
the subspace $\M^{\perp}$.  It follows that the map 
\begin{equation*}\label{zon32}
\sigma:\M^*\rightarrow \X^*/\M^{\perp}
\end{equation*}
given by
\begin{equation}\label{zon33}
\sigma\bigl(\bm^*\bigr) = \bz^* + \M^{\perp}
\end{equation}
is well defined.  Moreover, the map (\ref{zon33}) is an isometric isomorphism (see \cite[Theorem 4.9]{rudin1991})
from $\M^*$ onto $\X^*/\M^{\perp}$.  In view of (\ref{zon31}), for each point $\bm^*$ in $\M^*$ the 
norm of $\bm^*$ is given by
\begin{align}\label{zon34}
\begin{split}
\|\bm^*\|_{\M^*} &= \sup\big\{\langle \bm, \bm^*\rangle : \|\bm\|_{\M} \le 1\big\}\\
                           &= \inf\big\{\|\bz^* + \bwy^*\|_{\X^*} : \bwy^* \in \M^{\perp}\big\},
\end{split}
\end{align}
where $\bz^*$ in $\X^*$ extends $\bm^*$.

Now let $T:\R^N\rightarrow \X$ be an injective linear transformation.  We define a norm $\delta:\R^N\rightarrow [0, \infty)$ by
\begin{equation*}\label{pre9}
\delta(\bxi) = \|T(\bxi)\|_{\X}.
\end{equation*}
Because $T$ is injective, it follows that the linear transformation $T$ is an isometry
from $(\R^N, \delta)$ into $(\X, \|\ \|_{\X})$.
As in (\ref{zon1}) the dual norm $\delta^*:\R^N \rightarrow [0, \infty)$ is determined by
\begin{equation*}\label{pre10}
\delta^*(\bbeta) = \sup\{\langle \bxi, \bbeta\rangle : \delta(\bxi) \le 1\}.
\end{equation*}
The unit ball in $(\R^N, \delta)$, and the dual unit ball in $(\R^N, \delta^*)$, are given by 
\begin{equation}\label{pre11}
B = \{\bxi\in R^N: \delta(\bxi) \le 1\},\quad\text{and}\quad B^* = \{\bbeta\in R^N: \delta^*(\bbeta) \le 1\}.
\end{equation}
We recall that the adjoint transformation $T^*:\X^*\rightarrow \R^N$ is the unique linear transformation such that
\begin{equation}\label{pre12}
\langle \bxi, T^*(\bx^*)\rangle = \langle T(\bxi), \bx^*\rangle
\end{equation}
for all $\bxi$ in $\R^N$ and all $\bx^*$ in $\X^*$.
For our purposes it is useful to have an alternative description of $\delta^*$ and $B^*$ in terms of the adjoint 
transformation $T^*$.

\begin{lemma}\label{lem1}  The dual norm $\delta^*:\R^N \rightarrow [0, \infty)$ is given by
\begin{equation}\label{pre13}
\delta^*(\bbeta) = \inf\{\|\bx^*\|_{\X^*} : \bx^*\in \X^*\ \text{and}\ T^*(\bx^*) = \bbeta\},
\end{equation}
and the dual unit ball $B^*$ is
\begin{equation}\label{pre14}
B^* = \{T^*(\bx^*): \bx^*\in \X^*\ \text{and}\ \|\bx^*\|_{\X^*} \le 1\}.
\end{equation}
\end{lemma}

\begin{proof}  Let $\M = T(\R^N)$ denote the range of $T$, so that $\M\subseteq \X$ is a closed
linear subspace of dimension $N$.  Obviously $T$ induces an isometric isomorphism $U:\R^N\rightarrow \M$ by
$U(\bxi) = T(\bxi)$.  Then the adjoint map $U^*:\M^*\rightarrow \R^N$ is an isometric isomorphism with
respect to the norm $\delta^*$ on $\R^N$.  That is, if $\bbeta$ is in $\R^N$ and $\bm^*$ is the unique point in $\M^*$
such that $U^*(\bm^*) = \bbeta$, then
\begin{align*}\label{pre15}
\begin{split}
\delta^*(\bbeta) &= \sup\big\{\langle \bxi, \bbeta\rangle: \delta(\bxi) \le 1\big\}\\
                 &= \sup\big\{\langle \bxi, U^*(\bm^*)\rangle: \|U(\bxi)\|_{\M} \le 1\big\}\\
                 &= \sup\big\{\langle U(\bxi), \bm^*\rangle: \|U(\bxi)\|_{\M} \le 1\big\}\\
                 &= \sup\big\{\langle \bm, \bm^*\rangle: \|\bm\|_{\M} \le 1\big\}.
\end{split}
\end{align*} 
If $\bz^*$ in $\X^*$ extends the linear functional $\bm^*$, then from (\ref{zon34}) we get
\begin{equation}\label{pre16}
\delta^*(\bbeta) = \inf\big\{\|\bz^* + \bwy^*\|_{\X^*} : \bwy^* \in \M^{\perp}\big\}.
\end{equation}

The linear transformation $T$ can be written as the composition $T = i\circ U$, where
\begin{equation*}\label{pre17}
\R^N \buildrel U \over \longrightarrow \M \buildrel i \over \longrightarrow \X.
\end{equation*}
It follows that the adjoint map $T^*$ is the composition $T^* = U^*\circ i^*$, where
\begin{equation*}\label{pre18}
\X^* \buildrel i^* \over \longrightarrow \M^* \buildrel U^* \over \longrightarrow \R^N.
\end{equation*}
As $T$ is injective, the adjoint transformation $T^*$ is surjective (see \cite[Theorem 2.1]{carothers}).
Hence there certainly exists a point $\bz^*$ in $\X^*$ such that $T^*(\bz^*) = \bbeta$.  Then writing
$\bbeta = U^*\circ i^*(\bz)$ and $i^*(\bz) = \bm^*$, we find that $\bm^*$ is the unique point in $\M^*$ such 
that $U^*(\bm^*) = \bbeta$.  Therefore the linear functional $\bz^*$ does extend $\bm^*$ to $\X^*$.

Finally, the closed subspace $\M^{\perp}$ is also the kernel of $T^*$ (see \cite[Theorem 4.12]{rudin1991}).
As $T^*(\bz^*) = \bbeta$ we find that
\begin{align}\label{pre19}
\begin{split}
\big\{\bz^* + \bwy^*: \bwy^* \in \M^{\perp}\big\} &= \big\{\bz^* + \bwy^*: T^*(\bwy^*) = \bo\big\}\\
						 &= \big\{\bx^*: T^*(\bx^*) = \bbeta\big\}.
\end{split}
\end{align}
The identity (\ref{pre13}) follows now by combining (\ref{pre16}) and (\ref{pre19}).  Then (\ref{pre14})
is immediate from (\ref{pre13}).
\end{proof}

It will be convenient at this point to write $(X, \A, \nu)$ for a measure space.  That is, $X$ is a nonempty set,
$\A$ is a $\sigma$-algebra of subsets of $X$, and $\nu$ is a measure defined on the subsets in $\A$.  To avoid degenerate situations,
we assume that $\A$ contains subsets of positive $\nu$-measure.  We will
apply Lemma \ref{lem1} to the real Banach space $L^1(X, \A, \nu)$
and its dual space $L^{\infty}(X, \A, \nu)$.  Let $F_1, F_2, \dots , F_N$ be a collection
of $\R$-linearly independent elements from $L^1(X, \A, \nu)$.  By abuse of language we refer to
$F_1, F_2, \dots , F_N$ as functions rather than as equivalence classes of functions equal $\nu$-almost everywhere.  We adopt
a similar abuse of language with respect to other Banach spaces of measurable functions.  

Using the functions $F_1, F_2, \dots , F_N$, we define an injective linear transformation $T:\R^N\rightarrow L^1(X, \A, \nu)$ by setting
\begin{equation}\label{pre25}
T(\bxi) = \sum_{n=1}^N \xi_n F_n(x).
\end{equation}
We find that the adjoint transformation $T^*:L^{\infty}(X, \A, \nu) \rightarrow \R^N$ is given by $T^*(g) = \bbeta$, where $g$ is a function
in $L^{\infty}(X, \A, \nu)$ and $\bbeta$ is the vector in $\R^N$ having co-ordinates $\eta_n$, for $n = 1, 2, \dots , N$, given by
\begin{equation}\label{pre26}
\eta_n =  \int_X F_n(x) g(x)\ \dnu(x).
\end{equation}
Alternatively, if we write
\begin{equation*}\label{pre27}
\langle f, g \rangle = \int_X f(x) g(x)\ \dnu(x),
\end{equation*}
then $T^*:L^{\infty}(X, \A, \nu)\rightarrow \R^N$ is the unique linear transformation such that
\begin{equation}\label{pre28}
\langle \bxi, T^*(g) \rangle = \langle T(\bxi), g\rangle
\end{equation}
for all vectors $\bxi$ in $\R^N$ and all functions $g(x)$ in $L^{\infty}(X, \A, \nu)$.  

As before we write
\begin{equation}\label{pre29}
\delta(\bxi) = \|T(\bxi)\|_1 = \int_X~\Bigl|\sum_{n=1}^N \xi_n F_n(x)\Bigr|\ \dnu(x)
\end{equation}
for the norm in $\R^N$ induced by $T$, and
\begin{equation}\label{pre30}
B = \big\{\bxi \in \R^N: \|T(\bxi)\|_1 \le 1\big\}
\end{equation}
for the corresponding unit ball.  From Lemma \ref{lem1} we conclude that the dual norm is given by
\begin{equation*}\label{pre31}
\delta^*(\bbeta) = \inf\big\{\|g\|_{\infty}: g\in L^{\infty}(X, \A, \nu)\ \text{and}\ T^*(g) = \bbeta\big\},
\end{equation*}
and the dual unit ball is
\begin{equation}\label{pre32}
B^* = \big\{T^*(g): g\in L^{\infty}(X, \A, \nu)\ \text{and}\ \|g\|_{\infty} \le 1\big\}.
\end{equation}
The following lemma shows that the dual unit ball $B^*$ defined by (\ref{pre32}) is a zonoid.

\begin{lemma}\label{lem2}  Let $T:\R^N\rightarrow L^1(X, \A, \nu)$ be the injective linear transformation defined 
by {\rm (\ref{pre25})}.  Then there exists a unique, finite, even measure $\rho$, defined on the $\sigma$-algebra of Borel 
subsets of $\Sigma_{N-1}$, such that
\begin{equation}\label{pre33}
\|T(\bxi)\|_1 = \hh \int_{\Sigma_{N-1}} \bigl|\langle \bxi, \bbeta\rangle\bigr|\ \drho(\bbeta)
\end{equation}
for all vectors $\bxi$ in $\R^N$.
\end{lemma}

\begin{proof}  Let $\vphi(x)$ in $L^1(X, \A, \nu)$ be defined by
\begin{equation*}\label{pre34}
\vphi(x) = \big\{|F_1(x)|^2 + |F_2(x)|^2 + \cdots + |F_N(x)|^2\big\}^{\h}.
\end{equation*}
Then let $A$ in $\A$ be the measurable subset on which $\vphi(x)$ is positive.  As the functions
$F_1(x), F_2(x), \dots , F_N(x)$ are
$\R$-linearly independent elements of $L^1(X, \A, \nu)$, it is obvious that $\nu(A)$ is not zero.  Define a measurable function 
$\psi:A\rightarrow \Sigma_{N-1}$ by
\begin{equation*}\label{pre35}
\psi(x) = \vphi(x)^{-1}\bigl(F_n(x)\bigr).
\end{equation*}
Let $C(\Sigma_{N-1})$ be the algebra of continuous real valued functions defined on $\Sigma_{N-1}$ 
equipped with the supremum norm.  We define a positive linear functional $\Phi:C(\Sigma_{N-1})\rightarrow \R$ by
\begin{equation}\label{pre36}
\Phi(f) = 2 \int_A \vphi(x) f\bigl(\psi(x)\bigr)\ \dnu(x).
\end{equation}
It is clear that
\begin{equation*}\label{pre37}
\bigl|\Phi(f)\bigr| \le K \|f\|_{\infty},\quad\text{where}\quad K = \int_A \vphi(x)\ \dnu(x),
\end{equation*}
and therefore $f\mapsto \Phi(f)$ is continuous.  It follows from the Riesz representation theorem (see \cite[Theorems 2.14 and 2.17]{rudin1987})
that there exists a finite measure $\trho$, defined on the $\sigma$-algebra of Borel subsets of $\Sigma_{N-1}$, such that 
\begin{equation}\label{pre38}
\Phi(f) = \int_{\Sigma_{N-1}} f(\bbeta)\ \dtrho(\bbeta)
\end{equation}
for each function $f$ in $C(\Sigma_{N-1})$.  In particular, if $\bxi$ is a vector in $\R^N$, then
\begin{equation}\label{pre39}
\bbeta \mapsto \bigl|\langle \bxi, \bbeta\rangle\bigr|
\end{equation}
defines a function in $C(\Sigma_{N-1})$.  By applying (\ref{pre36}) and (\ref{pre38}) with this choice of $f$, we arrive at the identity
\begin{align}\label{pre40}
\begin{split}
\hh \int_{\Sigma_{N-1}} \bigl|\langle \bxi, \bbeta\rangle\bigr|\ \dtrho(\bbeta) 
		&= \int_A \vphi(x) \bigl|\langle \bxi, \psi(x)\rangle\bigr|\ \dnu(x)\\
		&= \int_A~\Bigl|\sum_{n=1}^N \xi_n F_n(x)\Bigr|\ \dnu(x)\\
		&= \|T(\bxi)\|_1.
\end{split}
\end{align}
Now let $\rho$ denote the even part of $\trho$.  That is, $\rho$ is the measure defined on Borel subsets $E$ of $\Sigma_{N-1}$ by
\begin{equation*}\label{pre41}
2 \rho(E) = \trho(E) + \trho(-E). 
\end{equation*}
As each of the functions (\ref{pre39}) is clearly even, (\ref{pre40}) implies that
\begin{equation*}\label{pre42}
\|T(\bxi)\|_1 = \hh \int_{\Sigma_{N-1}} \bigl|\langle \bxi, \bbeta\rangle\bigr|\ \drho(\bbeta)
\end{equation*}
for all vectors $\bxi$ in $\R^N$.  This verifies the identity (\ref{pre33}).

Finally, the assertion that $\rho$ is the {\it unique} even measure that satisfies the identity (\ref{pre33}) follows from
\cite[Theorem 1.4]{schneider1983}.
\end{proof}

\section{The volume of a zonoid}

In this section our main result is a formula for the $N$-dimensional volume of the zonoid $B^*$ defined in (\ref{pre32}).  Formulas for the
volume of a zonoid in terms of the even measure $\rho$ that occurs in the representation (\ref{pre33}) are known and
discussed in \cite{schneider1983}.  Here we establish a different formula for the volume of $B^*$ which depends on the functions
$F_1, F_2, \dots , F_N$.  Our result generalizes the formula (\ref{zon20}) for the volume of a zonotope obtained by McMullen in 
\cite{McMullen1984}. 

Using $F_1, F_2, \dots , F_N$, we define a function
\begin{equation*}\label{pre50}
\Delta: X^N \rightarrow \R
\end{equation*}
by setting
\begin{equation}\label{pre51}
\Delta(\bx) = \Delta(F_1, F_2, \dots , F_N; \bx) = \det\bigl(F_l(x_n)\bigr),
\end{equation}
where $l = 1, 2, \dots , N$ indexes rows, $n = 1, 2, \dots , N$ indexes columns, and $\bx$ denotes a point in the product space $X^N$ with 
co-ordinates $x_1, x_2, \dots , x_N$.  Let $\A^N$ denote the product $\sigma$-algebra
of subsets of $X^N$, and let $\nu^N$ denote the corresponding product measure defined on the $\sigma$-algebra $\A^N$.  We expand the function
$\Delta(\bx)$ as a sum over the symmetric group $S_N$ so that
\begin{equation}\label{pre52}
\Delta(\bx) = \sum_{\pi\in S_N} \sgn(\pi) \prod_{l=1}^N F_l(x_{\pi(l)}).
\end{equation}
It follows from (\ref{pre52}) that $\bx\mapsto \Delta(\bx)$ is $\A^N$-measurable because it is expressed as a sum in which each 
term is obviously $\A^N$-measurable.  The function $\bx\mapsto \Delta(\bx)$ is also integrable on $X^N$ with respect to the product measure
$\nu^N$.  This is a consequence of the bound
\begin{align}\label{pre53}
\begin{split}
\int_{X^N} \bigl|\Delta(\bx)\bigr|\ \dnu^N(\bx) &= \int_{X^N} \Bigl| \sum_{\pi\in S_N} \sgn(\pi) \prod_{l=1}^N F_l(x_{\pi(l)})\Bigr|\ \dnu^N(\bx)\\
	&\le \sum_{\pi\in S_N} \int_{X^N} \prod_{l=1}^N \bigl|F_l(x_{\pi(l)})\bigr|\ \dnu^N(\bx)\\
	&= N! \prod_{l=1}^N \|F_l\|_1.
\end{split}
\end{align}
We note that there is equality in the inequality (\ref{pre53}) whenever the integrable functions $F_1, F_2, \dots , F_N$ have disjoint support.

\begin{theorem}\label{thm3}  The volume of the zonoid $B^*$ defined by {\rm (\ref{pre32})} is given by
\begin{equation}\label{pre54}
\Vol_N(B^*) = \frac{2^N}{N!} \int_{X^N} \bigl|\Delta(F_1, F_2, \dots , F_N; \bx)\bigr|\ \dnu^N(\bx).
\end{equation}
\end{theorem}

Our proof of Theorem \ref{thm3} will require three lemmas.  We recall that a {\it simple} function $F:X\rightarrow \R$ is $\A$-measurable
and takes only finitely many distinct values.
 
\begin{lemma}\label{lem3}  Suppose that $F_1, F_2, \dots , F_N$ are $\R$-linearly independent simple functions in $L^1(X, \A, \nu)$.
Then there exists an $N\times M$ real matrix 
\begin{equation}\label{pre74}
A = \bigl(\ba_1\ \ba_2\ \cdots\ \ba_M\bigr) = (a_{lm}),
\end{equation}
and a partition 
\begin{equation*}\label{pre75}
X = E_0\cup E_1\cup E_2\cup \cdots \cup E_M
\end{equation*}
of $X$ into disjoint $\A$-measurable subsets, such that
\begin{itemize}
\item[(i)] the matrix $A$ has rank $N$, and no column of $A$ is identically zero,
\item[(ii)] for each integer $l = 1, 2, \dots , N$ the identity
\begin{equation}\label{pre76}
F_l(x) = \sum_{m=1}^M a_{lm} \chi_{E_m}(x)
\end{equation}
holds for $\nu$-almost all points $x$ in $X$, where $\chi_{E_m}(x)$ denotes the characteristic function of the subset $E_m$,
\item[(iii)] for each integer $m = 1, 2, \dots , M$ we have $0 < \nu(E_m) < \infty$.
\end{itemize}
\end{lemma}

\begin{proof}  Because each of the functions $F_1, F_2, \dots , F_N$ is a simple function, the map
\begin{equation}\label{pre80}
x \mapsto \left(\begin{matrix} F_1(x) \cr 
                                                      F_2(x) \cr
                                                      \cdots  \cr
                                                      F_N(x) \cr\end{matrix}\right) 
\end{equation}
takes finitely many values in $\R^N$.  Therefore we may partition $X$ into finitely many disjoint, $\A$-measurable subsets
$E_0, E_1, E_2, \dots , E_M$ such that (\ref{pre80}) is constant on the subset $E_m$ for each $m = 0, 1, 2, \dots , M$.   Clearly we
may assume that the map (\ref{pre80}) takes the value $\bo$ on the subset $E_0$, where $E_0$ may be the empty set. 
Then for $m = 1, 2, \dots , M$ we select a nonzero (column) vector $\ba_m$ in $\R^N$ such that 
\begin{equation*}\label{pre81}
\left(\begin{matrix} F_1(x) \cr 
                                                      F_2(x) \cr
                                                      \cdots  \cr
                                                      F_N(x) \cr\end{matrix}\right) = \ba_m\quad\text{whenever}\quad x \in E_m.
\end{equation*}
We write (\ref{pre74}) for the corresponding real $N\times M$ matrix, where $\ba_1, \ba_2, \dots , \ba_M$ denote the 
columns of the matrix $A$.  If we also write $\chi_{E_m}(x)$ for the characteristic function of the $\A$-measurable subset 
$E_m$, then each simple function $F_l(x)$ can be represented as
\begin{equation}\label{pre83}
F_l(x) = \sum_{m=1}^M a_{lm} \chi_{E_m}(x).
\end{equation}  
If one or more of the subsets $E_m$ occurring in the sum (\ref{pre83}) has $\nu$-measure zero, then we can remove that
term from the sum.  This changes each of the functions $F_l$ on a set of $\nu$-measure zero, and so does not change the 
element that $F_l$ represents in the Banach space $L^1(X, \A, \nu)$.  Thus we may assume without loss of generality that
$0 < \nu(E_m)$ for each $m = 1, 2, \dots , M$.  For each $m = 1, 2, \dots , M$ there exists an integer $l$ so that $a_{lm} \not= 0$.
As the corresponding simple function $F_l(x)$ is $\nu$-integrable on $X$, we must also have $\nu(E_m) < \infty$.

By hypothesis the simple functions $F_1, F_2, \dots , F_N$ are $\R$-linearly independent.  It follows that 
if $\bxi \not=\bo$ is in $\R^N$, then
\begin{equation*}\label{pre84}
\sum_{l=1}^N \xi_l F_l(x) = \sum_{m=1}^M \Big\{\sum_{l=1}^N \xi_l a_{lm}\Big\} \chi_{E_m}(x)
\end{equation*}
is not zero in $L^1(X, \A, \nu)$.  This shows that the rows of the matrix $A$ are $\R$-linearly independent, hence $A$ has rank $N$.
\end{proof}

\begin{lemma}\label{lem4}  Suppose that $F_1, F_2, \dots , F_N$ are $\R$-linearly independent functions in $L^1(X, \A, \nu)$ and 
$\epsilon > 0$.  Then there exist $\R$-linearly independent simple functions $s_1, s_2, \dots , s_N$ defined on $X$, and $\R$-linearly 
independent simple functions $t_1, t_2, \dots , t_N$ defined on $X$, such that
\begin{equation}\label{pre60}
\sum_{l=1}^N \|F_l - s_l\|_1 < \epsilon,\quad\text{and}\quad \sum_{l=1}^N \|F_l - t_l\|_1 < \epsilon,
\end{equation}
and the inequalities
\begin{equation}\label{pre61}
\Big\|\sum_{l=1}^N \xi_l s_l\Big\|_1 \le \Big\|\sum_{l=1}^N \xi_l F_l\Big\|_1\le \Big\|\sum_{l=1}^N \xi_l t_l\Big\|_1 
\end{equation}
hold for all vectors $\bxi$ in $\R^N$.
\end{lemma}

\begin{proof}  Let $C_1$ be a positive constant such that $\|F_l\|_1 \le C_1$ for each $l = 1, 2, \dots , N$.
Because the functions $F_1, F_2, \dots , F_N$ are $\R$-linearly independent, there exists a positive constant $3C_2$
such that
\begin{equation*}\label{pre62}
3C_2 |\bxi|_{\infty} \le \Big\|\sum_{l=1}^N \xi_l F_l \Big\|_1
\end{equation*}
for all vectors $\bxi$ in $\R^N$.  Let $0 < \eta \le C_2$, and then let $u_1, u_2, \dots , u_N$ be simple functions on $X$ such that
\begin{equation}\label{pre63}
\sum_{l=1}^N \|F_l - u_l\|_1 < \eta.
\end{equation}
It follows that
\begin{align*}\label{pre64}
\begin{split}
3C_2 |\bxi|_{\infty} &\le \Big\|\sum_{l=1}^N \xi_l\bigl(F_l - u_l\bigr)\Big\|_1 + \Big\|\sum_{l=1}^N \xi_l u_l\Big\|_1\\
                   &\le \eta |\bxi|_{\infty} + \Big\|\sum_{l=1}^N \xi_l u_l\Big\|_1,
\end{split}
\end{align*}
and therefore
\begin{equation}\label{pre65}
2C_2 |\bxi|_{\infty} \le \Big\|\sum_{l=1}^N \xi_l u_l\Big\|_1
\end{equation}
for all vectors $\bxi$ in $\R^N$.  In particular, (\ref{pre65}) implies that the simple functions $u_1, u_2, \dots , u_N$ are $\R$-linearly
independent.  We also get
\begin{align*}\label{pre66}
\begin{split}
\Big\|\sum_{l=1}^N \xi_l u_l\Big\|_1 &\le \Big\|\sum_{l=1}^N \xi_l (u_l - F_l)\Big\|_1 + \Big\|\sum_{l=1}^N \xi_l F_l\Big\|_1\\
                                     &\le \eta |\bxi|_{\infty} + \Big\|\sum_{l=1}^N \xi_l F_l\Big\|_1\\
                                     &\le \eta (2C_2)^{-1} \Big\|\sum_{l=1}^N \xi_l u_l\Big\|_1 + \Big\|\sum_{l=1}^N \xi_l F_l\Big\|_1,                                                                  
\end{split}
\end{align*}
and therefore
\begin{equation}\label{pre67}
\bigl(1 - \eta (2C_2)^{-1}\bigr) \Big\|\sum_{l=1}^N \xi_l u_l\Big\|_1 \le \Big\|\sum_{l=1}^N \xi_l F_l\Big\|_1
\end{equation}
for all vectors $\bxi$ in $\R^N$.  In a similar manner we find that
\begin{equation}\label{pre68}
\Big\|\sum_{l=1}^N \xi_l F_l\Big\|_1 \le \bigl(1 + \eta (2C_2)^{-1}\bigr) \Big\|\sum_{l=1}^N \xi_l u_l\Big\|_1
\end{equation}
for all vectors $\bxi$ in $\R^N$.  Thus we define
\begin{equation*}\label{pre69}
s_l(x) = \bigl(1 - \eta (2C_2)^{-1}\bigr) u_l(x),\quad\text{and}\quad t_l(x) = \bigl(1 + \eta (2C_2)^{-1}\bigr)u_l(x),
\end{equation*}
for $l = 1, 2, \dots , N$.  It follows that $s_1, s_2, \dots , s_N$ are linearly independent simple functions defined on $X$,
$t_1, t_2, \dots , t_N$ are linearly independent simple functions defined on $X$, and the inequality (\ref{pre61}) holds for all vectors $\bxi$
in $\R^N$.

To complete the proof we observe that
\begin{align}\label{pre70}
\begin{split}
\sum_{l=1}^N \|F_l - s_l\|_1 &\le \sum_{l=1}^N \|F_l - u_l\|_1 + \eta (2C_2)^{-1} \sum_{l=1}^N \|u_l\|_1\\
	                     &\le \eta + \eta (2C_2)^{-1} \Bigl(\eta + \sum_{l=1}^N \|F_l\|_1\Bigr)\\
	                     &\le \eta \bigl(1 + (2C_2)^{-1}\bigr)(C_2 + C_1N)\\
	                     &< \epsilon,
\end{split}
\end{align}
provided $\eta$ is sufficiently small.  An identical bound applies to the sum on the left of (\ref{pre70}) but with
$s_n$ replaced by $t_n$.  This verifies (\ref{pre60}).
\end{proof}

\begin{lemma}\label{lem5}  Suppose that $F_1, F_2, \dots , F_N$ and
$G_1, G_2, \dots , G_N$ belong to $L^1(X, \A, \nu)$.  Let $C_1$ be a positive constant such that
$\|F_n\|_1 \le C_1$ and $\|G_n\|_1 \le C_1$ hold for each $n = 1, 2, \dots , N$.
Then we have
\begin{align}\label{pre71}
\begin{split}
\int_{X^N} \bigl| \Delta(F_1&, F_2, \dots , F_N; \bx) - \Delta(G_1, G_2, \dots , G_N; \bx)\bigr|\ \dnu^N(\bx) \\
	&\le N!~C_1^{N-1} \sum_{l=1}^N \|F_l - G_l\|_1.
\end{split}
\end{align}
\end{lemma}

\begin{proof}  We use the identity
\begin{align*}\label{pre72}
\begin{split}
\Delta(F_1&, F_2, \dots , F_N; \bx) - \Delta(G_1, G_2, \dots , G_N; \bx)\\
	&= \sum_{l=1}^N \Delta(G_1, \dots , G_{l-1}, F_l, F_{l+1}, \dots , F_N; \bx)\\ 
	&\qquad\qquad - \sum_{l=1}^N \Delta(G_1, G_2, \dots , G_{l-1}, G_l, F_{l+1}, \dots , F_N; \bx)\\
	&= \sum_{l=1}^N \Delta(G_1, \dots , G_{l-1}, F_l - G_l, F_{l+1}, \dots , F_N; \bx)
\end{split}
\end{align*}
and the bound (\ref{pre53}).  In this way we obtain the inequality
\begin{equation*}\label{pre73}
\begin{split}
\int_{X^N} \bigl|&\Delta(F_1, F_2, \dots , F_N; \bx) - \Delta(G_1, G_2, \dots , G_N; \bx)\bigr|\ \dnu^N(\bx) \\
	&\le \sum_{l=1}^N \int_{X^N} \bigl| \Delta(G_1, \dots , G_{l-1}, F_l - G_l, F_{l+1}, \dots , F_N; \bx)\bigr|\ \dnu^N(\bx) \\
	&\le N!~\sum_{l=1}^N \big\{\|G_1\|_1 \cdots \|G_{l-1}\|_1 \|F_l - G_l\|_1 \|F_{l+1}\|_1 \cdots \|F_N\|_1\big\}\\
	&\le N!~C_1^{N-1} \sum_{l=1}^N \|F_l - G_l\|_1,
\end{split}
\end{equation*}
which proves the lemma.
\end{proof}

\begin{proof}[Proof of Theorem \ref{thm3}]
We assume to begin with that $F_1, F_2, \dots , F_N$ are $\R$-linearly independent simple functions in $L^1(X, \A, \nu)$.
Let $A = (a_{lm})$ denote the $N\times M$ matrix (\ref{pre74}), and let 
\begin{equation*}\label{pre85}
X = E_0\cup E_1\cup E_2\cup \cdots \cup E_M
\end{equation*}
be the partition of $X$ into disjoint $\A$-measurable subsets which satisfy the conclusions of Lemma \ref{lem3}.  Using these objects we define
a second $N\times M$ matrix 
\begin{equation}\label{pre86}
D = \bigl(a_{lm} \nu(E_m)\bigr).
\end{equation}
As $0 < \nu(E_m) < \infty$ for each $m = 1, 2, \dots, M$, we conclude that the matrix $D$ also has rank $N$.

Let $g$ belong to $L^{\infty}(X, \A, \nu)$ and satisfies $\|g\|_{\infty} \le 1$.  Then it follows from (\ref{pre32}) and (\ref{pre76}) that
\begin{align}\label{pre87}
\begin{split}
T^*(g) &= \bigl(\langle F_l, g\rangle\bigr)\\ 
	&= \Bigl(\sum_{m=1}^M a_{lm} \langle \chi_{E_m}, g\rangle \Bigr)\\
	&= \Bigl(\sum_{m=1}^M a_{lm} \nu(E_m) w_m\Bigr)\\
	&= D\bw
\end{split}
\end{align}
is the image of $g$ in $B^*$, where $\bw = (w_m)$ is the vector in $\R^M$ determined by
\begin{equation}\label{pre88}
w_m = \nu(E_m)^{-1} \int_{E_m} g(x)\ \dnu(x).
\end{equation}
Because $\|g\|_{\infty} \le 1$, it is obvious that $\|\bw\|_{\infty} \le 1$.  For an arbitrary vector $\bw$ in $\R^M$ with $\|\bw\|_{\infty} \le 1$ we 
can clearly select a function $g$ in $L^{\infty}(X, \A, \nu)$ such that $\|g\|_{\infty} \le 1$ and the co-ordinates of $\bw$ are given by (\ref{pre88}).
It follows that $B^*$ as defined by (\ref{pre32}), is also given by
\begin{equation}\label{pre89}
B^* = \big\{D\bw : \bw\in\R^M\ \text{and}\ \|\bw\|_{\infty} \le 1\big\}.
\end{equation}
Then from McMullen's formula (\ref{zon20}) we get
\begin{align}\label{pre90}
\begin{split}
\Vol_N\bigl(B^*\bigr) &= 2^N \sum_{|I| = N} \bigl|\det D_I\bigr|\\
	&= 2^N \sum_{|I| = N} \bigl|\det A_I\bigr|\prod_{m\in I} \nu(E_m).
\end{split}
\end{align}

For $\bx$ in $X^N$ let $Z(\bx)$ denote the $M\times N$ matrix
\begin{equation}\label{pre91}
Z(\bx) = \bigl(\chi_{E_m}(x_n)\bigr),
\end{equation}
where $m = 1, 2, \dots , M$ indexes rows and $n = 1, 2, \dots , N$ indexes columns.  Then (\ref{pre76}) leads to the matrix identity
\begin{equation*}\label{pre92}
\bigl(F_l(x_n)\bigr) = A Z(\bx).
\end{equation*}
Hence by the Cauchy-Binet identity we have
\begin{equation}\label{pre93}
\det\bigl(F_l(x_n)\bigr) = \sum_{|I| = N} \det A_I \det{_IZ(\bx)},
\end{equation}
where $A_I$ is the $N\times N$ submatrix of $A$ with columns indexed by $I$, and $_IZ(\bx)$ is the $N\times N$ submatrix of $Z(\bx)$
with rows indexed by $I$.  For each subset 
\begin{equation*}\label{pre94}
I = \{m_1 < m_2 < \cdots < m_N\} \subseteq \{1, 2, \dots , M\}
\end{equation*}
with cardinality $N$, let $\E(I)\subseteq X^N$ denote
the subset
\begin{equation}\label{pre95}
\E(I) = \bigcup_{\pi\in S_N(I)}\big\{E_{\pi(m_1)} \times E_{\pi(m_2)} \times \cdots \times E_{\pi(m_N)}\big\} \subseteq X^N,
\end{equation}
where $S_N(I)$ is symmetric group on the elements of the set $I$.  From (\ref{pre91}) and elementary properties of the determinant
we find that the map 
\begin{equation*}\label{pre100}
\bx\mapsto \det{_IZ(\bx)}
\end{equation*}
is supported on the subset $\E(I)$, and this map takes the value $\pm 1$ at each point of $\E(I)$.  We also note that
\begin{equation}\label{pre101}
\big\{\E(I): I\subseteq \{1, 2, \dots , M\}\ \text{and}\ |I| = N\big\}
\end{equation}
is a collection of {\it disjoint} $\A^N$-measurable subsets of $X^N$.  Then from (\ref{pre93}), (\ref{pre95}) and (\ref{pre101}), we obtain 
the identity
\begin{align}\label{pre102}
\begin{split}
\int_{X^N} \bigl|\Delta(F_1, F_2&, \dots , F_N; \bx)\bigr|\ \dnu^N(\bx)\\
	&= \sum_{|I| = N} \bigl|\det A_I\bigr| \int_{\E(I)} \bigl|\det{_IZ(\bx)}\bigr|\ \dnu^N(\bx)\\
	&= \sum_{|I| = N} \bigl|\det A_I\bigr| \nu^N\bigl(\E(I)\bigr)\\
	&= N! \sum_{|I| = N} \bigl|\det A_I\bigr| \prod_{m\in I} \nu(E_m).
\end{split}
\end{align}
By combining (\ref{pre90}) and (\ref{pre102}) we obtain the desired formula (\ref{pre54}), but under the assumption that
$F_1, F_2, \dots F_N$ are simple functions.

Now assume that $F_1, F_2, \dots , F_N$ are arbitrary $\R$-linearly independent functions in $L^1(X, \A, \nu)$ and let $0 < \epsilon \le 1$.
By Lemma \ref{lem4} there exist $\R$-linearly independent simple functions $s_1, s_2, \dots , s_N$ defined on $X$, and $\R$-linearly 
independent simple functions $t_1, t_2, \dots , t_N$ defined on $X$, which satisfy the conclusions (\ref{pre60}) and (\ref{pre61}) of
that result.   Write
\begin{equation*}\label{pre103}
C_1 = 2 \max\big\{\|F_1\|_1, \|F_2\|_1, \dots , \|F_N\|_1\big\},
\end{equation*}
so that
\begin{equation*}\label{pre104}
\max\big\{\|F_1\|_1, \dots , \|F_N\|_1, \|s_1\|_1, \dots , \|s_N\|_1, \|t_1\|_1, \dots , \|t_N\|_1\big\} \le C_1.
\end{equation*}
Let
\begin{equation*}\label{pre105}
B_s = \big\{\bxi\in \R^N: \big\| \xi_1s_1 + \xi_2s_2 + \cdots + \xi_Ns_N\big\|_1 \le 1\big\},
\end{equation*}
and
\begin{equation*}\label{pre106}
B_t = \big\{\bxi\in \R^N: \big\| \xi_1t_1 + \xi_2t_2 + \cdots + \xi_Nt_N\big\|_1 \le 1\big\},
\end{equation*}
be the corresponding unit balls.  From the inequality (\ref{pre61}) we conclude that
\begin{equation*}\label{pre107}
B_t \subseteq B \subseteq B_s,
\end{equation*}
and therefore 
\begin{equation}\label{pre108}
B_s^* \subseteq B^* \subseteq B_t^*.
\end{equation}
In particular, (\ref{pre108}) implies that
\begin{equation}\label{pre109}
\Vol_N\bigl(B_s^*\bigr) \le \Vol_N\bigl(B^*\bigr) \le \Vol_N\bigl(B_t^*\bigr).
\end{equation}
As $s_1, s_2, \dots , s_N$ are simple functions, by the case already considered we have
\begin{equation}\label{pre110}
\Vol\bigl(B_s^*\bigr) = \frac{2^N}{N!} \int_{X^N} \bigl|\Delta(s_1, s_2, \dots , s_N; \bx)\bigr|\ \dnu^N(\bx),
\end{equation}
and similarly
\begin{equation}\label{pre111}
\Vol\bigl(B_t^*\bigr) = \frac{2^N}{N!} \int_{X^N} \bigl|\Delta(t_1, t_2, \dots , t_N; \bx)\bigr|\ \dnu^N(\bx).
\end{equation}
Using (\ref{pre71}), (\ref{pre108}) and (\ref{pre110}), we obtain the inequality
\begin{align}\label{pre120}
\begin{split}
\frac{2^N}{N!} &\int_{X^N} \bigl|\Delta(F_1, F_2, \dots , F_N; \bx)\bigr|\ \dnu^N(\bx) - \Vol_N(B^*)\\
	&\le \frac{2^N}{N!}\Big\{\int_{X^N} \bigl|\Delta(F_1, F_2, \dots , F_N; \bx)\bigr|\ \dnu^N(\bx)\\ 
	&\qquad\qquad\qquad\qquad	- \int_{X^N} \bigl|\Delta(s_1, s_2, \dots , s_N; \bx)\bigr|\ \dnu^N(\bx)\Big\}\\
	&\le \frac{2^N}{N!} \int_{X^N} \bigl|\Delta(F_1, F_2, \dots , F_N; \bx) - \Delta(s_1, s_2, \dots , s_N; \bx)\bigr|\ \dnu^N(\bx)\\
	&\le (2 C_1)^N \sum_{n=1}^N \|F_n - s_n\|_1\\
	&\le (2 C_1)^N \epsilon.
\end{split}
\end{align}
Then the analogous inequality
\begin{equation}\label{pre121}
\Vol_N(B^*) - \frac{2^N}{N!} \int_{X^N} \bigl|\Delta(F_1, F_2, \dots , F_N; \bx)\bigr|\ \dnu^N(\bx) \le (2C_1)^N \epsilon
\end{equation}
can be established in a similar manner by using (\ref{pre111}).  The identity (\ref{pre54}) in the general case follows now by
combining (\ref{pre120}) and (\ref{pre121}).
\end{proof}

\section{Application of Minkowski's second theorem}

Again we suppose that $F_1, F_2, \dots , F_N$ are $\R$-linearly independent functions in the real Banach space $L^1(X, \A, \nu)$.
We write $B$ for the unit ball in $\R^N$ defined by (\ref{pre30}), and $B^*$ for the dual unit ball defined by (\ref{pre32}).  Lemma \ref{lem2}
shows that $B^*$ is a zonoid.  The inequality (\ref{zon5}) has been established by S.~Reisner \cite{reisner1985} in
this case, and therefore we have the inequality
\begin{equation}\label{st1}
\frac{4^N}{N!} \le \Vol_N(B) \Vol_N(B^*)
\end{equation}
for the volume product.  From (\ref{st1}) and the formula for the volume of a zonoid established in Theorem \ref{thm3}, we conclude that
\begin{equation}\label{st2}
2^N \le \Vol_N(B) \int_{X^N} \bigl|\Delta(F_1, F_2, \dots , F_N; \bx)\bigr|\ \dnu^N(\bx).
\end{equation}
Let $0 < \mu_1 \le \mu_2 \le \cdots \le \mu_N < \infty$ denote the successive minima of the convex symmetric set $B$ with
respect to the integer lattice $\Z^N$.  We recall (see \cite[Chapter VIII]{cassels1971} for further details) that $\mu_n$ is defined by
\begin{equation*}\label{st3}
\mu_n = \inf\big\{r > 0: rB\cap\Z^N\ \text{contains $n$ linearly independent points}\big\}.
\end{equation*}
If we write $T:\R^N\rightarrow L^1(X, \A, \nu)$ for the linear transformation defined by (\ref{pre25}), then it follows \cite[Lemma 1,
Chapter VIII]{cassels1971} that there exist linearly independent points $\bb_1, \bb_2, \dots , \bb_N$ in the integer lattice $\Z^N$
such that
\begin{equation}\label{st4}
\|T(\bb_n)\|_1 = \mu_n\quad\text{for each}\quad n = 1, 2, \dots , N.
\end{equation}
Minkowski's theorem on successive minima \cite[Theorem V, Chapter VIII]{cassels1971} asserts that the successive minima
satisfy the fundamental inequality
\begin{equation}\label{st5}
\mu_1 \mu_2 \cdots \mu_N \Vol_N(B) \le 2^N.
\end{equation}
Then the inequalities (\ref{st2}) and (\ref{st5}) imply that
\begin{equation}\label{st6}
\mu_1 \mu_2 \cdots \mu_N \le \int_{X^N} \bigl|\Delta(F_1, F_2, \dots , F_N; \bx)\bigr|\ \dnu^N(\bx).
\end{equation}
We reformulate these results in the following theorem.

\begin{theorem}\label{thm4}
Let $F_1, F_2, \dots , F_N$ be $\R$-linearly independent functions in $L^1(X, \A, \nu)$ and let $T:\R^N\rightarrow L^1(X, \A, \nu)$
be the linear transformation defined by
\begin{equation*}\label{st10}
T(\bxi) = \sum_{n=1}^N \xi_n F_n(x).
\end{equation*}
Then there exist linearly independent points $\bb_1, \bb_2, \dots , \bb_N$ in the integer lattice $\Z^N$ such that
\begin{equation}\label{st11}
\prod_{n=1}^N \|T(\bb_n)\|_1 \le \int_{X^N} \bigl|\Delta(F_1, F_2, \dots , F_N; \bx)\bigr|\ \dnu^N(\bx).
\end{equation}
Moreover, the $\Z$-module generated by $\bb_1, \bb_2, \dots, \bb_N$ has index at most $N!$ in $\Z^N$.
\end{theorem}

\begin{proof}  The inequality (\ref{st11}) plainly follows from (\ref{st4}) and (\ref{st6}).  The last assertion of the theorem 
is a well known corollary of Minkowski's theorem (see \cite[Corollary, Theorem V, Chapter VIII]{cassels1971}).
\end{proof}

Suppose that $U = (u_{mn})$ is an $N\times N$ real matrix and $G_1, G_2, \dots , G_N$ are functions
in $L^1(X. \A, \nu)$ determined for each $m = 1, 2, \dots , N$ by
\begin{equation}\label{st12}
G_m(x) = \sum_{n = 1}^N u_{mn} F_n(x).
\end{equation}
Then at each point $\bx$ in $X^N$ the identity
\begin{equation}\label{st13}
\Delta(G_1, G_2, \dots , G_N; \bx) = (\det U) \Delta(F_1, F_2, \dots , F_N; \bx)
\end{equation}
follows immediately from the definition (\ref{pre51}).  If the matrix $U$ has integer entries and $\det U = \pm 1$,
then the $\Z$-module generated by the functions $G_1, G_2, \dots , G_N$ is obviously equal to the $\Z$-module generated
by $F_1, F_2, \dots , F_N$.  In this case it follows from (\ref{st13}) that the value of the integral
\begin{equation}\label{st14}
\int_{X^N} \bigl|\Delta(F_1, F_2, \dots , F_N; \bx)\bigr|\ \dnu^N(\bx)
\end{equation}  
depends only on the $\Z$-module in $L^1(X, \A, \nu)$ generated by $F_1, F_2, \dots , F_N$, and does not depend
on the choice of generators.

By a {\it lattice} in $L^1(X, \A, \nu)$ we understand a discrete $\Z$-module $\eL$ in $L^1(X, \A, \nu)$ that is
generated by a finite collection of $\R$-linearly independent functions.  If $\eL$ is a lattice and has rank $N$, and if
$F_1, F_2, \dots , F_N$ is a collection of $\R$-linearly independent functions in $L^1(X, \A, \nu)$ that generate $\eL$, then we define
\begin{equation*}\label{st15}
I(\eL) = \int_{X^N} \bigl|\Delta(F_1, F_2, \dots , F_N; \bx)\bigr|\ \dnu^N(\bx).
\end{equation*}
By our previous remarks the value of the integral $I(\eL)$ is well defined because it does not depend on the choice 
of generators.      

\begin{corollary}\label{cor3}
Let $\eL$ be a lattice in $L^1(X, \A, \nu)$ of rank $N$.  Then there exist linearly independent functions $H_1, H_2, \dots , H_N$
in $\eL$ such that
\begin{equation}\label{st16}
\prod_{n=1}^N \|H_n\|_1 \le I(\eL).
\end{equation}
Moreover, the $\Z$-module generated by $H_1, H_2, \dots , H_N$ has index at most $N!$ in the lattice $\eL$.
\end{corollary}

\begin{proof}  We apply Theorem \ref{thm4} to a basis for $\eL$, and then we let $H_n = T(\bb_n)$.
\end{proof}

\section{Heights on finitely generated subgroups}

In this section we return to the situation discussed in the introduction.  Let $k$ be an algebraic number field of degree $d$ over 
$\Q$, $v$ a place of $k$ and $k_v$ the completion of $k$ at $v$.  We select two absolute values from the place $v$.  The 
first is denoted by $\|\ \|_v$ and defined as follows:
\begin{itemize}
\item[(i)] if $v|\infty$ then $\|\ \|_v$ is the unique absolute value on $k_v$ that extends the usual absolute value on $\Q_{\infty} = \R$,
\item[(ii)] if $v|p$ then $\| \ \|_v$ is the unique absolute value on $k_v$ that extends the usual $p$-adic absolute value on $\Q_p$.
\end{itemize}
The second absolute value is denoted by $|\ |_v$ and defined by $|x|_v=\|x\|_v^{d_v/d}$ for all $x$ in 
$k_v$, where $d_v = [k_v:\Q_v]$ is the local degree.  If $\alpha$ belongs to the multiplicative group $k^{\times}$
of nonzero elements in $k$, then these absolute values satisfy the product formula
\begin{equation}\label{prod1}
\sum_v \log |\alpha|_v = 0,
\end{equation}
where the sum on the left of (\ref{prod1}) is over all places $v$ of $k$, and only finitely many terms in the sum are nonzero.
If $\alpha$ is in $k^{\times}$ then the {\it absolute, logarithmic Weil height} (or simply the {\it height}) of $\alpha$ is defined by
\begin{equation}\label{ht1}
h(\alpha) = \sum_v \log^+|\alpha|_v.
\end{equation}
From (\ref{prod1}) and (\ref{ht1}) we obtain the identity
\begin{equation}\label{ht2}
2 h(\alpha) = \sum_v~\bigl|\log |\alpha|_v\bigr|,
\end{equation}
where $|\ |$ (an absolute value without a subscript) is the usual archimedean absolute value on $\R$. 
It can be shown that the value of the sum in (\ref{ht1}) does not depend on the field $k$ that contains $\alpha$.  Thus the height
is well defined on the multiplicative group $\oq$, and the height is constant on each coset of
the torsion subgroup $\Tor\bigl(\oq\bigr)$.  Therefore it is well defined on the quotient $\G = \oQt$.  As noted in the introduction, $\G$ is
a vector space (written multiplicatively) over the field $\Q$ of rational numbers, and $\alpha \mapsto 2h(\alpha)$ is a norm
on $\G$.  

Let $Y$ denote the set of all places $y$ of the field $\oQ$.  Let $k\subseteq \oQ$ be an algebraic 
number field such that $k/\Q$ is a finite Galois extension.  At each place $v$ of $k$ we write
\begin{equation}\label{un0}
Y(k,v) = \{y\in Y: y|v\}
\end{equation}
for the subset of places of $Y$ that lie over $v$.  Clearly we can express $Y$ as the disjoint union
\begin{equation}\label{un1}
Y = \bigcup_v~Y(k,v),
\end{equation}
where the union is over all places $v$ of $k$.  In \cite[section 2]{all2009} the authors show that each subset $Y(k, v)$
can be expressed as an inverse limit of finite sets.  This determines a totally disconnected, compact, Hausdorff topology
in $Y(k, v)$.  Then it follows from (\ref{un1}) that $Y$ is a totally disconnected, locally compact, Hausdorff space.  The
topology induced in $Y$ does not depend on the number field $k$.  It is also shown in \cite[section 3]{all2009} that
for each finite Galois extension $k/\Q$ the absolute Galois group $\Aut(\oQ/k)$ acts transitively and continuously on the elements of 
each compact, open subset $Y(k, v)$.  Moreover, (see \cite[Theorem 4]{all2009}) there exists a regular measure
 $\lambda$ defined on the Borel subsets of $Y$ that is positive
on nonempty open sets, finite on compact sets, and satisfies the identity $\lambda(\tau E) = \lambda (E)$ for all automorphisms
$\tau$ in $\Aut(\oQ/k)$ and all Borel subsets $E$ of $Y$.  The measure $\lambda$ is unique up to a positive multiplicative
constant.  In \cite[Theorem 5]{all2009} we established a choice of $\lambda$ such that
\begin{equation}\label{un2}
\lambda\bigl(Y(k, v)\bigr) = \dfrac{[k_v: \Q_v]}{[k: \Q]}
\end{equation}
for each finite Galois extension $k/\Q$ and each place $v$ of $k$.  More generally, if $k/\Q$ is a finite,
but not necessarily Galois, extension then the identity (\ref{un2}) continues to hold.  

\begin{lemma}\label{lem6}  Let $\lambda$ be the unique regular measure defined on the $\sigma$-algebra $\B$ of Borel subsets of $Y$
which satisfies $\lambda(\tau E) = \lambda (E)$ for all automorphisms $\tau$ in $\Aut(\oQ/k)$ and all Borel subsets $E$ of $Y$, and
which satisfies {\rm (\ref{un2})} for all finite Galois extensions $k/\Q$ and all places $v$ of $k$.  Then the identity 
{\rm (\ref{un2})} holds for all finite, but not necessarily Galois, extensions $k/\Q$. 
\end{lemma}

\begin{proof}  Let $l/\Q$ be a finite Galois extension
such that $\Q \subseteq k \subseteq l$.  Then for each place $v$ of $k$ we have the disjoint union
\begin{equation}\label{un3}
Y(k, v) = \bigcup_{w|v} Y(l, w),
\end{equation}
where the union on the right of (\ref{un3}) is over all places $w$ of the finite Galois extension $l$ such that $w|v$.  As (\ref{un2}) is
already known to hold for finite Galois extensions, we find that
\begin{align*}\label{un4}
\begin{split}
\lambda\bigl(Y(k, v)\bigr) &= \sum_{w|v} \lambda\bigl(Y(l, w)\bigr)\\
                                            &= \sum_{w|v}  \dfrac{[l_w: \Q_w]}{[l: \Q]}\\
                                            &= \sum_{w|v}  \dfrac{[l_w: k_v]}{[l: k]}\dfrac{[k_v: \Q_v]}{[k:\Q]}\\
                                            &=  \dfrac{[k_v: \Q_v]}{[k:\Q]}.                                           
\end{split}
\end{align*}
This verifies the lemma.
\end{proof}

If $y$ is a place in $Y(k,v)$ we select an absolute value $\|\ \|_y$ from $y$ such that the restriction of $\|\ \|_y$ to $k$ is equal to $\|\ \|_v$.  As the
restriction of $\|\ \|_v$ to $\Q$ is one of the usual absolute values on $\Q$, it follows that this choice
of the normalized absolute value $\|\ \|_y$ does not depend on $k$.  If $\alpha$ is a point in $\G$ then we 
associate $\alpha$ with the continuous, compactly supported function
\begin{equation}\label{ht11}
y\mapsto f_{\alpha}(y) = \log \|\alpha\|_y
\end{equation}
defined on the locally compact Hausdorff space $Y$, (see \cite[equation (1.9)]{all2009}).  Each function (\ref{ht11})
belongs to the real Banach space $L^1(Y, \B, \lambda)$, where $\B$ is the $\sigma$-algebra of Borel subsets of $Y$, and
$\lambda$ is the normalized measure on $\B$ that is invariant with respect to the natural Galois action (see
\cite[Theorem 4]{all2009}).  Moreover, the map
\begin{equation}\label{ht12}
\alpha\mapsto f_{\alpha}
\end{equation}
is a linear transformation from the $\Q$-vector space $\G$ into the real Banach space $L^1(Y, \B, \lambda)$.  Let
\begin{equation*}\label{ht13}
\F = \big\{f_{\alpha}(y): \alpha\in \G\big\}
\end{equation*}
denote the image of $\G$ under this linear map.  Then $\alpha\mapsto 2h(\alpha)$ is a norm on the $\Q$-vector space $\G$, and
$f_{\alpha}\mapsto \|f_{\alpha}\|_1$ is obviously a norm on the $\Q$-vector space $\F$.  With respect to these norms, 
the map $\alpha\rightarrow f_{\alpha}$ is a linear isometry from the vector space $\G$ (written multiplicatively) onto the vector 
space $\F$ (written additively).  It follows from the product formula (see \cite[equation (1.10)]{all2009}), that each function $f_{\alpha}$ 
in $\F$ belongs to the closed subspace
\begin{equation*}\label{ht14}
\X = \Big\{F\in L^1(Y, \B, \lambda): \int_{Y}F(y)\ \dla(y) = 0\Big\}.
\end{equation*}
Then Theorem \ref{thm2} in \cite{all2009} asserts that $\F$ is a dense subset of $\X$ with respect to the $L^1$-norm.

As $\alpha\mapsto f_{\alpha}$ is a linear isometry, we have
\begin{equation}\label{ht15}
h(\alpha) = \hh \int_Y \bigl|f_{\alpha}(y)\bigr|\ \dla(y) 
\end{equation}
at each point $\alpha$ in $\G$.  We wish to extend (\ref{ht15}) to finitely generated subgroups of $\G$, or equivalently to finitely
generated subgroups of $\F$.  To accomplish this we require the following lemma.

\begin{lemma}\label{lem7}  Suppose that $\alpha_1, \alpha_2, \dots , \alpha_N$ are $\Q$-linearly independent elements
of $\G$, then the corresponding functions $f_{\alpha_1}(y), f_{\alpha_2}(y), \dots , f_{\alpha_N}(y)$ are $\R$-linearly 
independent elements of $\X$.
\end{lemma} 

\begin{proof}  For each $n = 1, 2, \dots , N$, let $\gamma_n$ be a representative in $\oq$ of the coset $\alpha_n$ in $\G = \oQt$. 
Let $k\subseteq \oQ$ be an algebraic number field such that
\begin{equation*}\label{ht16}
\{\gamma_1, \gamma_2, \dots , \gamma_N\} \subseteq k^{\times}.
\end{equation*}
Then let $S$ be a finite subset of places of $k$ such that $S$ contains all archimedean places of $k$, and such that 
\begin{equation*}\label{ht17}
\{\gamma_1, \gamma_2, \dots , \gamma_N\} \subseteq U_S(k),
\end{equation*}
where 
\begin{equation*}\label{ht18}
U_S(k) = \{\beta \in k: |\beta|_v = 1\ \text{for all}\ v\notin S\}
\end{equation*}
denotes the multiplicative group of $S$-units in $k$.  Write $s$ for the cardinality of $S$ and assume that $s \ge 2$.
By the $S$-unit theorem (stated as \cite[Theorem 3.5]{narkiewicz1974}), there exist multiplicatively independent 
elements $\eta_1, \eta_2, \dots , \eta_{s-1}$ in $U_S(k)$ which form a fundamental system of $S$-units.  Write
\begin{equation}\label{ht19}
M = \bigl([k_v:\Q_v] \log \|\eta_r\|_v\bigr)
\end{equation}  
for the associated $(s - 1)\times s$ real matrix, where $r = 1, 2, \dots , s-1$ indexes rows and $v$ in $S$ indexes columns.  As the 
$S$-regulator does not vanish, the $s-1$ rows of the matrix $M$ are $\R$-linearly independent. 

Because the elements $\eta_1, \eta_2, \dots , \eta_{s-1}$ form a fundamental system of $S$-units, there exists an 
$N\times (s - 1)$ integer matrix 
\begin{equation}\label{ht20}
L = \bigl(l_{nr}\bigr),
\end{equation}
where $n = 1, 2, \dots , N$ indexes rows and $r = 1, 2, \dots , s-1$ indexes columns, such that
\begin{equation}\label{ht21}
\log \|\gamma_n\|_v = \sum_{r=1}^{s-1} l_{nr} \log \|\eta_r\|_v
\end{equation}
at each place $v$ in $S$.  Write $\Gamma$ for the $N\times s$ matrix
\begin{equation*}\label{ht22}
\Gamma = \bigl([k_v: \Q_v] \log \|\gamma_n\|_v\bigr),
\end{equation*}
where $n = 1, 2, \dots , N$ indexes rows and $v$ in $S$ indexes columns.  Then (\ref{ht21}) implies that $\Gamma = L M$.  By
hypothesis the rows of $\Gamma$ are $\Q$-linearly independent.  Therefore the rows of $L$ must be $\Q$-linearly independent.
Because $L$ has integer entries, it follows that the rows of $L$ are $\R$-linearly independent.  As the rows
of $M$ are $\R$-linearly independent, we find that the rows of $\Gamma = L M$ are also $\R$-linearly independent.  This proves the lemma. 
\end{proof} 

Let $\aA\subseteq \G$ be a finitely generated subgroup of rank $N$.
The map (\ref{ht12}) is an isometric isomorphism from the $\Q$-vector space $\G$ onto the $\Q$-vector space $\F$, and 
therefore we may identify $\aA$ with its image $\aA^{\prime}$ in $\F$.  If the elements $\alpha_1, \alpha_2, \dots , \alpha_N$
form a basis for $\aA$ as a $\Z$-module, then their images $f_{\alpha_1}(y), f_{\alpha_2}(y), \dots , f_{\alpha_N}(y)$ in $\F$ obviously 
form a basis for $\aA^{\prime}$ as a $\Z$-module in $\F$.  Then Lemma \ref{lem7} asserts that the functions 
$f_{\alpha_1}(y), f_{\alpha_2}(y), \dots , f_{\alpha_N}(y)$ are $\R$-linearly independent elements of the Banach space $\X$,
and therefore they generate a lattice in $\X$ as discussed in section 5.
Therefore we define the height of $\aA$ (or the height of $\aA^{\prime}$) by
\begin{equation}\label{ht25}
h(\aA) = 2^{-N} \int_{Y^N} \bigl|\Delta(f_{\alpha_1}, f_{\alpha_2}, \dots , f_{\alpha_N}; \bwy)\bigr|\ \dla^N(\bwy).
\end{equation}
Because of the identity (\ref{st13}), the integral on the right of (\ref{ht25}) does not depend on the choice of generators for the
subgroup $\aA$, and therefore $h(\aA)$ is well defined on the collection of finitely generated subgroups of $\G$.  Of course 
this height is (aside from the factor $2^{-N}$) a special case of the expression (\ref{st14}) which is well defined on the more 
general collection of finitely generated subgroups in $L^1(Y, \B, \lambda)$.  We note that if $\aA$ has rank $1$, then 
(\ref{ht25}) is the absolute logarithmic height (\ref{ht15}) applied to a generator of $\aA$.

\begin{proof}[Proof of Theorem \ref{thm1}]  Let $\alpha_1, \alpha_2, \dots , \alpha_N$ form a basis for the finitely generated
subgroup $\aA\subseteq \G$.  By Lemma \ref{lem7} the functions $f_{\alpha_1}, f_{\alpha_2}, \dots , f_{\alpha_N}$ are $\R$-linearly
independent elements of the Banach space $\X \subseteq L^1(Y, \B, \lambda)$.  Therefore we can apply Theorem \ref{thm4} with 
$F_n = f_{\alpha_n}$.  By that result there exist linearly independent points $\bb_1, \bb_2, \dots , \bb_N$ in the integer lattice 
$\Z^N$ which satisfy the inequality (\ref{st11}).  Write
\begin{equation*}\label{ht30}
\bigl(\bb_1\ \bb_2\ \cdots \bb_N\bigr) = (b_{mn})
\end{equation*}
for the corresponding $N\times N$ matrix having the vector $\bb_n$ as its $n$th column.  For $n = 1, 2, \dots , N$ let
$\beta_n$ in $\G$ be defined by
\begin{equation}\label{ht31}
f_{\beta_n}(y) = \sum_{m = 1}^N b_{mn} f_{\alpha_m}(y).
\end{equation}
Then (\ref{st11}) can be written as
\begin{equation}\label{ht32}
\prod_{n = 1}^N \Big\{\int_Y \bigl|f_{\beta_n}(y)\bigr|\ \dla(y)\Big\} 
	\le \int_{Y^N} \bigl|\Delta(f_{\alpha_1}, f_{\alpha_2}, \dots , f_{\alpha_N}; \bwy)\bigr|\ \dla^N(\bwy).
\end{equation}
The inequality (\ref{setup25}) follows from (\ref{ht32}) by applying identities (\ref{ht15}) and (\ref{ht25}) for the height.

The sublattice generated by the vectors $\bb_1, \bb_2, \dots , \bb_N$ has index at most $N!$ in $\Z^N$.   Then it follows from
(\ref{ht31}) that the subgroup $\bB \subseteq \aA$ generated by $\beta_1, \beta_2, \dots , \beta_N$ has the same index, and 
in particular has index at most $N!$.
\end{proof}

\section{Application of Siegel's Lemma}

Again we suppose that $\alpha_1, \alpha_2, \dots , \alpha_N$ are elements of $\G$, but in this section we
assume that these elements generate a multiplicative subgroup $\aA\subseteq \G$ with positive rank $M$, and we assume that
$1 \le M < N$.  Hence the $\Z$-module of linear dependencies
\begin{equation}\label{siegel0}
\eZ = \Big\{\bz \in \Z^N : \prod_{n=1}^N \alpha_n^{z_n} = 1\Big\}.
\end{equation}
has rank $L = N-M$.  We use the functions $f_{\alpha_1}(y), f_{\alpha_2}(y), \dots , f_{\alpha_N}(y)$ in $\F$, to define
an $M\times N$ matrix
\begin{equation}\label{siegel1}
U(\balpha; \bwy) = \bigl(f_{\alpha_n}(y_l)\bigr)
\end{equation}
with entries in $\F$, where $l = 1, 2, \dots , M$ indexes rows and $n = 1, 2, \dots , N$ indexes
columns.  Then we define a nonnegative function $Q:Y^M\rightarrow [0, \infty)$ by
\begin{align}\label{siegel2}
\begin{split}
Q(\balpha; \bwy) &= \Big\{\det\bigl(U(\balpha; \bwy) U(\balpha; \bwy)^T\bigr)\Big\}^{\h}\\
		 &= \Big\{\det\Bigl(\sum_{n=1}^N f_{\alpha_n}(y_l) f_{\alpha_n}(y_m)\Bigr)\Big\}^{\h}.
\end{split}
\end{align}
Using Hadamard's inequality (see \cite[Theorem 30]{hardy1934}) we find that
\begin{align}\label{siegel3}
\begin{split}
Q(\balpha; \bwy) &\le \prod_{m=1}^M \Big\{\sum_{n=1}^N \bigl(f_{\alpha_n}(y_m)\bigr)^2\Big\}^{\h}\\ 
                               &\le \prod_{m=1}^M \Big\{\sum_{n=1}^N \bigl|f_{\alpha_n}(y_m)\bigr|\Big\},
\end{split}
\end{align}
at each point $\bwy$ in $Y^M$.  By integrating both sides of (\ref{siegel3}) over $Y^M$, we get the bound
\begin{align}\label{siegel4}
\begin{split}
2^{-M} \int_{Y^M} Q(\balpha; \bwy)\ \dla^M(\bwy) 
	&\le 2^{-M} \prod_{m=1}^M \Big\{\sum_{n=1}^N \int_Y \bigl|f_{\alpha_n}(y_m)\bigr|\ \dla(y_m)\Big\}\\
	&= \Big\{\sum_{n=1}^N h(\alpha_n)\Big\}^M.
\end{split}
\end{align}
In particular (\ref{siegel4}) shows that the function $\bwy\mapsto Q(\bwy)$ is integrable 
on $Y^M$ with respect to the product measure $\lambda^M$.

\begin{theorem}\label{thm5}  Let $\alpha_1, \alpha_2, \dots , \alpha_N$ be
elements of $\G$ which generated a subgroup $\aA$ of positive rank $M$.
If $1 \le M < N$ then there exist $L = N-M$ linearly independent elements $\bz_1, \bz_2, \dots , \bz_L$ in the 
$\Z$-module $\eZ$ defined by {\rm (\ref{siegel0})}, such that
\begin{equation}\label{siegel8}
\bigg\{\prod_{l=1}^L |\bz_l|_{\infty}\bigg\} h(\aA) \le 2^{-M} \int_{Y^M} Q(\balpha; \bwy)\ \dla^M(\bwy).
\end{equation}
\end{theorem}

\begin{proof}  Let $\gamma_1, \gamma_2, \dots , \gamma_M$ form a basis for $\aA$.  Then the functions 
\begin{equation}\label{siegel9}
f_{\gamma_1}(y), f_{\gamma_2}(y) , \dots , f_{\gamma_M}(y)
\end{equation}
form a basis for the image of $\aA$ in $\F$.  Hence there exists an $M\times N$ matrix $A = (a_{mn})$ with integer entries such that
\begin{equation}\label{siegel12}
f_{\alpha_n}(y) = \sum_{m=1}^M a_{mn} f_{\gamma_m}(y)
\end{equation}
for each $n = 1, 2, \dots , N$.  And there exists an $N\times M$ matrix $B = (b_{nm})$ with integer entries such that
\begin{equation}\label{siegel13}
f_{\gamma_m}(y) = \sum_{n=1}^N b_{nm} f_{\alpha_n}(y)
\end{equation}
for each $m = 1, 2, \dots , M$.  A vector $\bz$ in $\Z^N$ belongs to $\eZ$ if and only if
\begin{align*}\label{siegel14}
\begin{split}
0 &= \sum_{n=1}^N z_n f_{\alpha_n}(y)\\
    &= \sum_{n=1}^N z_n \sum_{m=1}^M a_{mn} f_{\gamma_m}(y)\\
    &= \sum_{m=1}^M \Big\{\sum_{n=1}^N a_{mn} z_n \Big\} f_{\gamma_m}(y).
\end{split}
\end{align*}
Because the functions (\ref{siegel9}) form a basis, we find that
\begin{equation}\label{siegel15}
\eZ = \big\{\bz \in \Z^N: A \bz = \bo\big\}.
\end{equation}
By the basis form of Siegel's Lemma (see \cite[Theorem 2]{bombieri1983}), there exist $L = N-M$ linearly independent
vectors $\bz_1, \bz_2, \dots , \bz_L$ in $\eZ$ such that
\begin{equation}\label{siegel16}
\prod_{l=1}^L |\bz_l|_{\infty} \le D^{-1}\big\{\det AA^T\big\}^{\h},
\end{equation}
where $D$ is the greatest common divisor of the collection of all $M\times M$ subdeterminants of the matrix $A$.
Combining (\ref{siegel12}) and (\ref{siegel13}) leads to the identity $AB = \bone_M$, where $\bone_M$ denotes the 
$M\times M$ identity matrix.  As both $A$ and $B$ have integer entries, it follows from the Cauchy-Binet
identity (see \cite[Lemma 2, Chapter I]{cassels1971}) that $D = 1$.

Let
\begin{equation*}\label{siegel21}
V(\bgamma; \bwy) = \big(f_{\gamma_m}(y_l)\bigr)
\end{equation*}
denote the $M\times M$ matrix with entries in $\X$, where $l = 1, 2, \dots , M$ indexes rows and $m = 1, 2, \dots , M$ indexes
columns.  Then (\ref{siegel1}) and (\ref{siegel12}) imply that
\begin{equation*}\label{siegel22}
U(\balpha; \bwy) = V(\bgamma; \bwy) A,
\end{equation*}
and therefore
\begin{align*}\label{siegel23}
\begin{split}
Q(\balpha; \bwy)^2 &= \det\bigl(U(\balpha; \bwy) U(\balpha; \bwy)^T\bigr)\\
                                   &= \det V(\bgamma; \bwy) \det AA^T \det V(\bgamma; \bwy)^T\\
                                   &= \det AA^T \bigl(\det V(\bgamma; \bwy)\bigr)^2.
\end{split}
\end{align*}
It follows using (\ref{pre51}) that
\begin{equation}\label{siegel24}
Q(\balpha; \bwy) = \big\{\det AA^T\big\}^{\h} \bigl|\Delta(f_{\gamma_1}, f_{\gamma_2}, \dots , f_{\gamma_M}; \bwy)\bigr|.
\end{equation}
Integrating both sides of (\ref{siegel24}) over $Y^M$ with respect to the product measure $\lambda^M$ leads to the identity
\begin{align}\label{siegel25}
\begin{split}
2^{-M} \int_{Y^M} &Q(\balpha; \bwy)\ \dla^M(\bwy)\\ 
		&= 2^{-M} \big\{\det AA^T\big\}^{\h} 
                   \int_{Y^M} \bigl|\Delta(f_{\gamma_1}, f_{\gamma_2}, \dots , f_{\gamma_M}; \bwy)\bigr|\ \dla^M(\bwy)\\
		&= \big\{\det AA^T\big\}^{\h} h(\aA).
\end{split}
\end{align}
The inequality (\ref{siegel8}) follows now be combining (\ref{siegel16}) and (\ref{siegel25}).
\end{proof}

\begin{proof}[Proof of Theorem \ref{thm2}]  We apply Theorem \ref{thm1} to the subgroup $\aA \subseteq \G$
of rank $M$.  By that result there exist $M$ multiplicatively independent elements $\beta_1, \beta_2, \dots , \beta_M$ in
$\aA$ such that
\begin{equation}\label{siegel30}
h(\beta_1)h(\beta_2) \cdots h(\beta_M) \le h(\aA).
\end{equation}
We use the bound (\ref{siegel30}) on the left hand side of (\ref{siegel8}), and we use
(\ref{siegel4}) on the right hand side of (\ref{siegel8}).  In this way we arrive at the inequality (\ref{setup8}) in
the statement of Theorem \ref{thm2}.
\end{proof}

\begin{proof}[Proof of Corollary \ref{cor2}]  Let $\eta_1, \eta_2, \dots , \eta_{s-1}$ be multiplicatively independent
elements in $U_S(k)$ which form a fundamental system of $S$-units.  Then
\begin{equation}\label{siegel40}
\big\{f_{\eta_1}(y), f_{\eta_2}(y), \dots , f_{\eta_{s-1}}(y)\big\}
\end{equation}
is the image of this fundamental system in $\G_k$, and therefore (\ref{siegel40}) forms a basis 
for the group $\uU_S(k)$.  At each place $v$ of $k$ let $E_v = \{y\in Y: y|v\}$, and write $\chi_{E_v}(y)$ for
the characteristic function of the compact set $E_v$.  Then at each point $y$ in $Y$ we have the representation
\begin{equation}\label{siegel44}
f_{\eta_l}(y) = \sum_{v\in S} \log\|\eta_l\|_v \chi_{E_v}(y),\quad\text{for}\quad l = 1, 2, \dots , s-1.
\end{equation}
Clearly (\ref{siegel44}) is analogous to the representation (\ref{pre76}) in the statement
of Lemma \ref{lem3}, but now the disjoint subsets are 
\begin{equation*}\label{siegel45}
\big\{E_v: v\in S\big\},
\end{equation*}
and they are indexed by the places in $S$.  Let $A$ denote the $(s-1)\times s$ real matrix
\begin{equation}\label{siegel46}
A = \bigl(\log\|\eta_l\|_v \bigr)
\end{equation}
where $l = 1, 2, \dots , s-1$ indexes rows and $v$ in $S$ indexes columns.  For $\bwy$ in $Y^{(s-1)}$ let $Z(\bwy)$ denote
the $s\times (s-1)$ matrix
\begin{equation}\label{siegel47}
Z(\bwy) = \bigl(\chi_{E_v}(y_n)\bigr)
\end{equation}
where $v$ in $S$ indexes rows and $n = 1, 2, \dots , s-1$ indexes columns.

From (\ref{ht25}) we find that
\begin{equation}\label{siegel50}
h\bigl(\uU_S(k)\bigr) = 2^{1-s} \int_{Y^{s-1}} \bigl|\Delta(f_{\eta_1}, f_{\eta_2}, \dots , f_{\eta_{s-1}}; \bwy)\bigr|\ \dla^{(s-1)}(\bwy).
\end{equation}
We argue exactly as in our derivation (\ref{pre102}), and obtain the identity 
\begin{align}\label{siegel51}
\begin{split}
\int_{Y^{s-1}} \bigl|\Delta(f_{\eta_1}, f_{\eta_2},& \dots , f_{\eta_{s-1}}; \bwy)\bigr|\ \dla^{(s-1)}(\bwy)\\
     &= (s-1)! \sum_{\substack{I\subseteq S\\|I| = s-1}} \bigl|\det A_I\bigr| \prod_{v\in I} \lambda(E_v),
\end{split}
\end{align}
where the sum on the right of (\ref{siegel51}) runs over subsets $I \subseteq S$ with cardinality $s-1$.  Using Lemma \ref{lem6},
for each subset $I\subseteq S$ with cardinality $s-1$ we find that
\begin{align}\label{siegel52}
\begin{split}
\bigl|\det A_I\bigr| \prod_{v\in I} \lambda(E_v) &= [k: \Q]^{1-s} \bigl|\det A_I\bigr| \prod_{v\in I} [k_v: \Q_v]\\
                                             &= [k: \Q]^{1-s} \bigl|\det \bigl([k_v: \Q_v] \log \|\eta_l\|_v\bigr)_I\bigr|\\
                                             &= [k: \Q]^{1-s} \Reg_S(k),
\end{split}
\end{align}
where $\Reg_S(k)$ denotes the $S$-regulator.  Finally, we combine (\ref{siegel50}), (\ref{siegel51}) and (\ref{siegel52}).  In this
way we obtain the identity
\begin{equation}\label{siegel53}
h\bigl(\uU_S(k)\bigr) = \dfrac{(s-1)! \Reg_S(k)}{\bigl(2[k: \Q]\bigr)^{s-1}} \sum_{\substack{I\subseteq S\\|I| = s-1}} 1
                                      = \dfrac{s! \Reg_S(k)}{\bigl(2 [k: \Q]\bigr)^{s-1}}.
\end{equation}
This proves the corollary.
\end{proof}


\today
\end{document}